\documentclass[a4paper,11pt,reqno]{amsart}
\usepackage[latin1]{inputenc}
\usepackage[english]{babel}
\usepackage{amsmath, amsthm, amssymb, amsopn, amsfonts, amstext, stmaryrd, enumerate, color, hyperref, mathtools, graphicx, epstopdf}
\numberwithin{equation}{section}

\DeclareGraphicsExtensions{.eps}

\newcommand{\E}{\mathscr{E}}
\newcommand{\F}{\mathcal{F}}
\renewcommand{\H}{\mathcal{H}}

\newcommand{\R}{\mathbb{R}}

\newcommand{\T}{\mathcal{T}}
\newcommand{\Z}{\mathbb{Z}}
\newcommand{\W}{\mathscr{W}}
\newcommand{\loc}{{\rm loc}}
\newcommand{\dive}{\mbox{\normalfont div}}
\newcommand{\Hess}{{\mbox{\normalfont Hess}}}
\newcommand{\Mat}{{\mbox{\normalfont Mat}}}
\newcommand{\sgn}{{\mbox{\normalfont sgn}}}

\theoremstyle{plain}
\newtheorem{theorem}{Theorem}[section]
\newtheorem{proposition}[theorem]{Proposition}
\newtheorem{lemma}[theorem]{Lemma}

\theoremstyle{definition}
\newtheorem{remark}[theorem]{Remark}

\renewcommand{\le}{\leqslant}

\renewcommand{\ge}{\geqslant}

\usepackage{mathrsfs}

\begin{document}

\title[Singular, degenerate, anisotropic PDEs]{Monotonicity formulae 
and classification results \\ for singular, degenerate, anisotropic PDEs}

\author[Matteo Cozzi, Alberto Farina, Enrico Valdinoci]{
Matteo Cozzi${}^{(1,2)}$
\and
Alberto Farina${}^{(1)}$
\and 
Enrico Valdinoci${}^{(2,3)}$
}



\thanks{This paper has been supported by ERC grant 277749 ``EPSILON Elliptic Pde's and Symmetry of
Interfaces and Layers for Odd Nonlinearities''.}

\maketitle

\date{}

{\scriptsize \begin{center} (1) -- Laboratoire
Ami\'enois de Math\'ematique Fondamentale et Appliqu\'ee\\
UMR CNRS 7352, Universit\'e Picardie ``Jules Verne'' \\33 Rue St Leu, 
80039 Amiens (France).\\
\end{center}
\scriptsize \begin{center} (2) -- Dipartimento di Matematica
``Federigo Enriques''\\
Universit\`a
degli studi di Milano,\\
Via Saldini 50, I-20133 Milano (Italy).\\
\end{center}
\scriptsize \begin{center} (3) -- Weierstra{\ss}
Institut f\"ur Angewandte Analysis und Stochastik\\
Mohrenstra{\ss}e 39, D-10117 Berlin (Germany).
\end{center}
\medskip
\begin{center}
E-mail addresses: matteo.cozzi@unimi.it,
alberto.farina@u-picardie.fr,
enrico.valdinoci@wias-berlin.de
\end{center}
}
\bigskip

\begin{abstract}
We consider possibly degenerate and singular
elliptic equations in a possibly anisotropic medium.

We obtain monotonicity results for the energy density,
rigidity results for the solutions and classification results
for the singularity/degeneracy/anisotropy allowed.

As far as we know, these results are new even in the case of
non-singular and non-degenerate anisotropic equations.
\end{abstract}

\bigskip

\section{Introduction and main results}

\subsection{Description of the model and mathematical setting}

The goal of this paper is to consider partial differential
equations in a possibly anisotropic medium.

The interest in the study of
anisotropic media is twofold. First, at a purely mathematical level, the lack of isotropy reflects into a rich
geometric structure in which the basic objects of investigation
do not possess the usual Euclidean properties. Then, from
the point of view of concrete applications, anisotropic media
naturally arise in the study of crystals, see e.g.~\cite{C84}
and the references therein. The interplay
between the concrete physical problems and the
geometric structures is clearly discussed, for instance, in~\cite{T78, TCH92}. We also refer to Appendix~C in~\cite{CFV14}
for a simple physical application.

The equations that we consider in the present paper
have a variational structure
and they are of elliptic type, though the ellipticity
is allowed to be possibly singular or degenerate.

The forcing term only depends on the values of the solution,
i.e., in jargon,
the equation is quasilinear, and the elliptic operator
is constant along the level sets of the solution.
This feature imposes strong geometric restrictions
on the solution, and the purpose of this paper is to
better understand some of these properties.
\medskip

In this setting we present a variety of results
from different perspectives, such as:
\begin{itemize}
\item a monotonicity formula for the energy functional
(i.e., the energy of an anisotropic ball, suitably rescaled,
will be shown to be non-decreasing with respect to the size of
such ball);
\item a rigidity result of Liouville type (namely,
if the potential is integrable, then the solution needs
to be constant);
\item a precise classification of some of the assumptions
given in the literature, with concrete examples and some
simplifications.
\end{itemize}
\medskip

The formal mathematical notation introduces the solution~$u$
of an anisotropic equation driven by a possibly nonlinear
operator. The anisotropic term is encoded into a homogeneous
function~$H$, that will be often referred to as ``the anisotropy''.
The nonlinearity feature of the operator is given by a function~$B$
(e.g., the function~$B$ can be a power and produce an equation
of $p$-Laplace type). Also, the nonlinear source term
arises from a potential~$F$.
\medskip

More precisely,
given measurable set~$\Omega \subset \R^n$, with~$n \ge 2$,
we consider the Wulff type energy functional
\begin{equation} \label{Wen}
\W_\Omega(u) := \int_\Omega B(H(\nabla u(x))) - F(u(x)) \, dx,
\end{equation}
and the associated Euler-Lagrange equation
\begin{equation} \label{eleq}
\frac{\partial}{\partial x_i} 
\Big( B'( H(\nabla u) ) H_i(\nabla u) \Big) + F'(u) = 0.
\end{equation}

Here, the function~$B$ belongs to~$C^{3, \beta}_\loc((0, +\infty)) \cap C^1([0, +\infty))$, with~$\beta \in (0, 1)$, and is such that~$B(0) = B'(0) = 0$ and
\begin{equation} \label{Bpos}
B(t), B'(t), B''(t) > 0 \mbox{ for any } t > 0.
\end{equation}
Moreover,~$H$ is a positive homogeneous function of degree~$1$, of class~$C^{3, \beta}_\loc(\R^n \setminus \{ 0 \})$ and for which
\begin{equation} \label{Hpos}
H(\xi) > 0 \mbox{ for any } \xi \in \R^n \setminus \{ 0 \}.
\end{equation}
Using its homogeneity properties, we infer that~$H$ can be naturally extended to a continuous function on the whole of~$\R^n$ by setting~$H(0) = 0$. 
Moreover, the forcing term~$F$ is required to be~$C^{2, \beta}_\loc(\R)$.

In addition to these hypotheses we also assume one of the following conditions to hold:

\begin{enumerate}[(A)]
\item There exist~$p > 1$,~$\kappa \in [0, 1)$ and positive~$\gamma, \Gamma$ such that, for any~$\xi \in \R^n \setminus \{ 0 \}$,~$\zeta \in \R^n$,
$$
\left[ \Hess \,(B \circ H)(\xi) \right]_{i j} \zeta_i \zeta_j \ge \gamma {(\kappa + |\xi|)}^{p - 2} {|\zeta|}^2, \label{BHpell}
$$
and
$$
\sum_{i, j = 1}^n \left| \left[ \Hess \,(B \circ H)(\xi) \right]_{i j} \right| \le \Gamma {(\kappa + |\xi|)}^{p - 2}.
$$
\item The composition~$B \circ H$ is of class~$C^{3,\beta}_\loc(\R^n)$ 
and for any~$K>0$
there exist a positive constant~$\gamma$ such that,
for any~$\xi, \zeta \in \R^n$, with~$|\xi|\le K$, we have
\begin{equation*}
\left[ \Hess \,(B \circ H)(\xi) \right]_{i j} \zeta_i \zeta_j \ge \gamma 
\,{|\zeta|}^2.
\end{equation*}
\end{enumerate}

In~\cite[Appendix~A]{CFV14} we showed that hypothesis~(A) is fulfilled for instance by taking~$B(t) = t^p / p$ together with an~$H$ whose anisotropic
unit ball
\begin{equation} \label{Hball}
B^H_1 = \left\{ \xi \in \R^n : H(\xi) < 1 \right\}, 
\end{equation}
is uniformly convex, i.e. such that the principal curvatures of its boundary are bounded away from zero. Every anisotropy~$H$ having uniformly convex unit ball will be called \emph{uniformly elliptic}. We remark that, since the second fundamental form of~$\partial B_1^H$ at a point~$\xi \in \partial B_1^H$ is given by
$$
\mathrm{I\!I}_\xi(\zeta, \upsilon) = \frac{H_{i j}(\xi) \zeta_i \upsilon_j }{|\nabla H(\xi)|} \mbox{ for any } \zeta, \upsilon \in \nabla H(\xi)^\perp,
$$
as can bee seen for instance in~\cite[Appendix~A]{CFV14}, and being~$\partial B_1^H$ compact, the uniform ellipticity of~$H$ is equivalent to ask
\begin{equation} \label{Huniell}
H_{i j}(\xi) \zeta_i \zeta_j \ge \lambda |\zeta|^2 \mbox{ for any } \xi \in \partial B_1^H, \, \zeta \in \nabla H(\xi)^\perp,
\end{equation}
for some~$\lambda > 0$. Any positive~$\lambda$ for which~\eqref{Huniell} is satisfied will be said to be an \emph{ellipticity constant} for~$H$. Notice that, by homogeneity,~\eqref{Huniell} actually extends to
\begin{equation} \label{Huniell2}
H_{i j}(\xi) \zeta_i \zeta_j \ge \lambda |\xi|^{- 1} |\zeta|^2 \mbox{ for any } \xi \in \R^n \setminus \{ 0 \}, \, \zeta \in \nabla H(\xi)^\perp.
\end{equation}

\medskip
 
We associate to our solution~$u \in L^\infty(\R^n)$ the finite 
quantities
$$ u^* := \sup_{\R^n} u \ {\mbox{ and }} \
u_* := \inf_{\R^n} u,$$ 
and the gauge
\begin{equation} \label{cudef}
c_u := \sup \left\{ F(t) : t \in \left[ u_*, u^* \right] \right\}.
\end{equation}
Finally, for~$t \in \R$ we set
\begin{equation} \label{Gdef}
G(t) := c_u - F(t).
\end{equation}
Notice that such~$G$ is a non-negative function on the range of~$u$
and that putting it in place of~$-F$ in~\eqref{Wen} does not change at all the setting, once~$u$ is fixed.

\medskip

In the forthcoming Subsections~\ref{XX1}--\ref{XX2}, we give precise statements of our main results. We point out that, to the best of our knowledge, these results are new even in the case in which~$B(t)=t^2/2$ (i.e. even in the case in which the elliptic operator is non-singular and non-degenerate).

\subsection{A monotonicity formula}\label{XX1}

Monotonicity formulae are a classical topic in geometric variational analysis.
Roughly speaking, 
the idea of monotonicity formulae is that a suitably
rescaled energy functional in a ball possesses some
monotonicity properties with respect to the radius of the ball
(in our case, the situation is geometrically
more complicated, since the ball is non-Euclidean).

Of course, monotonicity formulae are important, since
they provide a quantitative information on the energy
of the problem; moreover, they often provide additional
information on the asymptotic behaviour of the solutions,
also in connection with blow-up and blow-down limits,
and they play a special role in rigidity and classification results.
\medskip

One of the main results of the present paper consists in 
a monotonicity formula for a suitable rescaled version
of the functional~\eqref{Wen}, over the family of sets indexed by~$R > 0$,
\begin{equation} \label{Wball}
W_R^H = W_R := \left\{ x \in \R^n : H^*(x) < R \right\},
\end{equation}
where, for~$x \in \R^n$,
\begin{equation} \label{H*def}
H^*(x) := \sup_{\xi \in S^{n - 1}} \frac{\langle x, \xi \rangle}{H(\xi)},
\end{equation}
is the dual function of~$H$. Notice that~$H^*$ is a positive homogeneous function of degree~$1$ and that it is at least of class~$C^2(\R^n \setminus \{ 0 \})$, as showed in Lemma~\ref{H*reg}
below. The set~$W_R$ is the so-called Wulff shape of radius~$R$ associated to~$H$. We refer to~\cite{CS09, WX11} for some basic properties of this set and to~\cite{T78} for a nice geometrical construction. The precise statement is given by

\begin{theorem} \label{monformthm}
Assume that one of the following conditions to be valid:
\begin{enumerate}[(i)]
\item Assumption~(A) holds and~$u \in L^\infty(\R^n) \cap W_\loc^{1, p}(\R^n)$ is a weak solution of~\eqref{eleq} in~$\R^n$;
\item Assumption~(B) holds and $u \in W^{1, \infty}(\R^n)$ weakly solves~\eqref{eleq} in~$\R^n$.
\end{enumerate}
If~$(i)$ is in force, assume in addition that~$H$ satisfies, for any~$\xi, x \in \R^n$,
\begin{equation} \label{FKweak}
\sgn \langle H(\xi) \nabla H(\xi), H^*(x) \nabla H^*(x) \rangle = \sgn \langle \xi, x \rangle.
\end{equation}
Then, the rescaled energy defined by
\begin{equation} \label{scaledWen}
\E(u; R) = \E(R) := \frac{1}{R^{n - 1}} \int_{W_R} B(H(\nabla u(x))) + G(u(x)) \, dx,
\end{equation}
for any~$R > 0$, is monotone non-decreasing.
\end{theorem}

Observe that when~$B(t)=t^2/2$ and~$H(x)=|x|$
(i.e. when the operator is simply the Laplacian and the equation
is isotropic), then 
the result of Theorem~\ref{monformthm}
reduces to the classical
monotonicity formula proved in~\cite{M89}.
Then, if~$H(x)=|x|$, the results of~\cite{M89}
were extended to the non-linear case in~\cite{CGS94}. 
Differently from the existing literature,
here we introduce the presence of a general non-Euclidean anisotropy~$H$ 
(also, we remove an unnecessary assumption on
the sign of~$F$).

We remark that the anisotropic term in the monotonicity formula
provides a number of geometric complications. Indeed,
in our case, the unit ball~$B_1^H$ is not Euclidean and it does
not coincide with its dual ball~$W_1^H$, and a point on the unit sphere
does not coincide in general with the normal to the sphere.

Also, we mention that Theorem~\ref{monformthm}
relies on the pointwise gradient estimate proved in~\cite[Theorem~1.1]{CFV14}.

\subsection{Geometric conditions on the anisotropy and classification results}

In the statement of the monotonicity formula the new condition~\eqref{FKweak} is assumed on~$H$. Here, we plan to shed some light on its origin and to better understand its implications.

First, we point out that this assumption comes as a weaker form of the more restrictive
\begin{equation} \label{FK}
\langle H(\xi) \nabla H(\xi), H^*(x) \nabla H^*(x) \rangle = \langle \xi, x \rangle,
\end{equation}
for any~$\xi, x \in \R^n$. To the authors' knowledge, this latter condition has been first introduced in~\cite{FK09} to recover the validity of the mean value property for~$Q$-harmonic functions, that are the solutions of the equation
\begin{equation} \label{Qu=0}
Q u := \frac{\partial}{\partial x_i} \Big( H(\nabla u) H_i(\nabla u) \Big) = 0.
\end{equation}
Notice that such solutions are the counterparts of harmonic functions in the anisotropic framework and that equation~\eqref{Qu=0} is 
a particular case of
our setting by taking~$B(t) = t^2/2$ and~$F = 0$.

Examples of homogeneous functions~$H$ for which~\eqref{FK} is valid are the norms displayed in~\eqref{HM}, as showed by Lemma~\ref{FKneclem}. Note that we do not assume~\eqref{FKweak} in case~$(ii)$ of Theorem~\ref{monformthm}. Indeed, hypothesis~(B) forces $H$ to be of the form~\eqref{HM}, as shown in~\cite[Appendix~B]{CFV14} (this can also be deduced from the forthcoming Theorem~\ref{T-RES}). In the next result we emphasize that anisotropies as
the one in~\eqref{HM} are actually the~\emph{only} ones which satisfy~\eqref{FK}.

\begin{theorem} \label{FKcharprop}
Let~$H \in C^1(\R^n \setminus \{ 0 \})$ be a positive homogeneous function of degree~$1$ satisfying~\eqref{Hpos}. Assume 
that its unit ball~$B_1^H$, as defined by~\eqref{Hball}, is
strictly convex. Then, condition~\eqref{FK} is equivalent to
asking~$H$ to be of the form
\begin{equation} \label{HM}
H_M(\xi) = \sqrt{\langle M \xi, \xi \rangle},
\end{equation}
for some symmetric and positive definite matrix~$M \in \Mat_n(\R)$.
\end{theorem}


{F}rom Theorem~\ref{FKcharprop},
it follows that assumption~\eqref{FK}
imposes some severe restrictions on the geometric structure of its 
unit ball, which is always an Euclidean
ellipsoid. A natural question is therefore
to understand in which sense
our condition~\eqref{FKweak} is more general. For this scope,
we will discuss condition~\eqref{FKweak}
in detail, by making concrete examples and obtaining a 
complete characterization
in the plane. Roughly speaking, the unit ball in the
plane under condition~\eqref{FKweak} can be constructed by
considering a curve in the first quadrant that satisfies a suitable,
explicit differential inequality, and then~\emph{reflecting}
this curve in the other quadrants (of course,
if higher regularity on the ball is required, this gives
further conditions on the derivatives of the curve at
the reflection points). The detailed characterization
of condition~\eqref{FKweak} in the plane is given by the following technical but operational result.

\begin{proposition} \label{2Dprop-i}
Let~$r: [0, \pi/2] \to (0, +\infty)$ be a given~$C^2$ function satisfying
\begin{equation} \label{rcond1-i}
r(\theta) r''(\theta) < 2 r'(\theta)^2 + r(\theta)^2 \mbox{ for a.a. } \theta \in \left[ 0, \frac{\pi}{2} \right],
\end{equation}
and
\begin{equation} \label{rcond2-i}
r(0) = 1, \qquad r(\pi/2) = r^*, \qquad r'(0) = r'(\pi/2) = 0,
\end{equation}
for some~$r^* \ge 1$. Consider the~$\pi$-periodic function~$\widetilde{r}: \R \to (0, +\infty)$ defined on~$[0, \pi]$ by
\begin{equation*}
\widetilde{r}(\theta) :=
\begin{dcases}
r(\theta) & \quad {\mbox{if }} 0 \le \theta \le \frac{\pi}{2}, \\
\frac{r^* \sqrt{r(\tau^{-1}(\theta))^2 + r'(\tau^{-1}(\theta))^2}}{
r(\tau^{-1}(\theta))^2} & \quad {\mbox{if }} \frac{\pi}{2} \le 
\theta \le \pi,
\end{dcases}
\end{equation*}
where~$\tau: [0, \pi/2] \to [\pi/2, \pi]$ is the bijective map given by
\begin{equation*}
\tau(\eta) = \frac{\pi}{2} + \eta - \arctan \frac{r'(\eta)}{r(\eta)}.
\end{equation*}
Then,~$\widetilde{r}$ is of class~$C^1(\R)$, the set
\begin{equation} \label{Cdef-i}
\left\{ (\rho \cos \theta, \rho \sin \theta) : \rho \in [0, \widetilde{r}(\theta)), \, \theta \in [0, 2 \pi] \right\},
\end{equation}
is strictly convex and its supporting function
$$
\widetilde{H}(\rho \cos \theta, \rho \sin \theta) := \frac{\rho}{\widetilde{r}(\theta)},,
$$
defined for~$\rho \ge 0$ and~$\theta \in [0, 2 \pi]$, satisfies~\eqref{FKweak}.

Furthermore, up to a rotation and a homothety of the plane~$\R^2$, any even positive $1$-homogeneous function~$H \in C^2(\R^2 \setminus \{ 0 \})$ satisfying~\eqref{Hpos}, having strictly convex unit ball~$B_1^H$ and for which condition~\eqref{FKweak} holds true is such that~$B_1^H$ is of the form~\eqref{Cdef-i}, for some positive~$r \in C^2([0, \pi/2])$ satisfying~\eqref{rcond1-i} and~\eqref{rcond2-i}. 
\medskip

In addition, if~$H \in C^{3,\alpha}_\loc(\R^2 \setminus \{ 0 \})$,
for some~$\alpha\in(0,1]$, we have that~$H$ is uniformly elliptic
and satisfies condition~\eqref{FKweak} 
if and only if~$r \in C^{3, \alpha}([0, \pi / 2])$, inequality~\eqref{rcond1-i} is satisfied at any~$\theta \in [0, \pi / 2]$ and
$$
r'' \left( \frac{\pi}{2} \right) = - \frac{r^* r''(0)}{1 - r''(0)}, \qquad
r''' \left( \frac{\pi}{2} \right) = - \frac{r^* r'''(0)}{(1 - r''(0))^3},
$$
hold along with~\eqref{rcond2-i}.
\end{proposition}

With this characterization, it is easy to construct examples
satisfying condition~\eqref{FKweak} whose corresponding ball
is not an Euclidean ellipsoid, see Remark~\ref{final}.

\subsection{A rigidity result}
As an application of Theorem~\ref{monformthm} we have the following Liouville-type result.

\begin{theorem} \label{liouthm}
Let~$H$ and~$u$ be as in Theorem~\ref{monformthm}. If
\begin{equation} \label{Gugrowth}
\int_{W_R} G(u(x)) \, dx = o(R^{n - 1}) \mbox{ as } R \rightarrow +\infty,
\end{equation}
then~$u$ is constant.

In particular, if~$G(u) \in L^1(\R^n)$, then~$u$ is constant.
\end{theorem}

We remark that Theorem~\ref{liouthm} is a sort of rigidity result.
The condition that~$G(u)$ has finite mass - or, more generally, that the mass has controlled growth - may be seen
as a prescription of the values of the solution at infinity
(at least, in a suitably averaged sense): 
the result of Theorem~\ref{liouthm} gives that the only solution 
that can satisfy such prescription is the trivial one.
In this spirit, Theorem~\ref{liouthm} may be seen as a variant
of the classical Liouville Theorem for harmonic functions
(set here in a nonlinear, anisotropic, singular or degenerate framework).
\medskip

\subsection{Equivalent conditions}\label{XX2}
We remark that the assumptions in~(A) and~(B)
that we made on the anisotropic and nonlinear part of the operator
are somehow classical in the literature, see e.g.~\cite{CGS94, CFV14}
and the references therein (roughly speaking, these conditions
are the necessary ones to obtain some regularity of the solutions
using, or adapting, the elliptic regularity theory).

In spite of their classical flavour, we think that in some cases these conditions
can be made more explicit or more concrete. For this, in
this paper we provide some equivalent characterizations. 
In particular, we will observe that condition~(B) puts
some important restrictions on the structure of the ambient medium, due to
the regularity requirement on the composition~$B \circ H$.
More precisely, the following result holds true:

\begin{theorem}\label{T-RES} 
Assumption~(A) is equivalent to 
\begin{enumerate}[(A)\ensuremath{'}] 
\item There exist~$p > 1$,~$\bar{\kappa} \in [0, 1)$ and 
positive~$\bar{\gamma}, \bar{\Gamma}, 
\lambda$ such that 
$$ H \mbox{ is uniformly elliptic with 
constant~$\lambda$}, $$ 
and 
\begin{align*} \bar{\gamma} (\bar{\kappa} + 
t)^{p - 2} t \le & B'(t) \le \bar{\Gamma} (\bar{\kappa} + t)^{p - 2} t, 
\\ \bar{\gamma} (\bar{\kappa} + t)^{p - 2} \le & B''(t) \le \bar{\Gamma} 
(\bar{\kappa} + t)^{p - 2}, \end{align*} 
for any~$t > 0$. 
\end{enumerate} 
Assumption~(B) is equivalent to 
\begin{enumerate}[(B)\ensuremath{'}]
\item The function~$B$ is of class~$C^{3, \beta}_\loc([0, +\infty))$, with~$B'''(0) = 0$,
\begin{equation} \label{B''0pos}
B''(0) > 0,
\end{equation}
and~$H$ is of the type~\eqref{HM}, for some~$M \in \Mat_n(\R)$ symmetric and positive definite.
\end{enumerate}
\end{theorem}

\subsection{Organization of the paper}
The rest of the paper is organized as follows.

In Section~\ref{auxsec} we gather several auxiliary lemmata, most of which are related to basic properties of the anisotropy~$H$. At the end of the section we also briefly comment on the regularity of the solutions of~\eqref{eleq}.

In Section~\ref{ABsec} we establish the equivalence of the two sets of conditions~(A)-(B) and~(A)\ensuremath{'}-(B)\ensuremath{'}, thus proving Theorem~\ref{T-RES}.

The proof of the main result of the paper, Theorem~\ref{monformthm}, is the content of Section~\ref{monformsec}. In the subsequent Section~\ref{liousec} we then deduce Theorem~\ref{liouthm} as a corollary of the monotonicity formula.

The last two sections deal with the characterizations of conditions~\eqref{FK} and~\eqref{FKweak}. In Section~\ref{char1sec} we address Theorem~\ref{FKcharprop}, while the following Section~\ref{char2sec} is devoted to the proof of Proposition~\ref{2Dprop-i}.

\section{Some auxiliary results} \label{auxsec}

We collect here some preliminary results which will be abundantly used in the forthcoming sections. Most of them are very well known results, so that we will not comment much on their proofs. Nevertheless, precise references will be given.

Every result in this section is clearly meant to be applied to the functions~$H$ and~$B$ above introduced. However, when possible we state them under slightly lighter hypotheses.

The first lemma provides three useful identities for the derivatives of positive homogeneous functions. We recall that, given~$d \in \R$, a function~$H: \R^n \setminus \{ 0 \} \to \R$ is said to be positive homogeneous of degree~$d$ if
$$
H(t \xi) = |t|^d H(\xi) \mbox{ for any } t > 0, \, \xi \in \R^n \setminus \{ 0 \}.
$$

\begin{lemma} \label{homain}
If~$H\in C^3(\R^n\setminus \{0\})$ is positive homogeneous of
degree~$1$, we have that
\begin{align}
\label{i} H_i(\xi) \xi_i & = H(\xi),\\
\label{ii} H_{ij}(\xi) \xi_i & = 0,\\
\label{iii} H_{ijk}(\xi) \xi_i & = -H_{jk}(\xi).
\end{align}
\end{lemma}

We refer to the Appendix of~\cite{FV14} for a proof. The second result of this section deals with the regularity up to the origin of both the anisotropy~$H$ and the composition~$B \circ H$.

\begin{lemma} \label{derBH0}
Let~$H\in C^1(\R^n\setminus\{0\})$ be a positive homogeneous function of degree $d$ admitting non-negative values and~$B \in C^1([0, +\infty))$, with~$B(0) = 0$. Assume that either $d > 1$ or $d = 1$ and $B'(0) = 0$. Then~$H$ can be extended by setting~$H(0):=0$ to a continuous function, such that~$B \circ H \in C^1(\R^n)$ and
$$
\partial_i (B \circ H)(0) = 0 = \lim_{x \rightarrow 0} B'(H(x)) H_i(x).
$$
\end{lemma}

A proof of Lemma~\ref{derBH0} can be found in~\cite[Section~2]{CFV14}. Next is a lemma which gathers some results on~$H$ and its dual~$H^*$.

\begin{lemma} \label{H*reg}
Let~$B \in C^2((0, +\infty))$ and~$H \in C^2(\R^n \setminus \{ 0 \})$. Assume~$B$ to satisfy~\eqref{Bpos}, the function~$H$ to be positive homogeneous of degree~$1$ satisfying~\eqref{Hpos} and~$\Hess(B \circ H)$ to be positive definite in~$\R^n \setminus \{ 0 \}$. Then, the ball~$B_1^H$ defined by~\eqref{Hball} is strictly convex.\\
Furthermore, the dual function~$H^*$ defined by~\eqref{H*def} is of class~$C^2(\R^n \setminus \{ 0 \})$, the formulae
\begin{equation} \label{CS2}
H^*(\nabla H(\xi)) = H(\nabla H^*(\xi)) = 1,
\end{equation}
hold true for any~$\xi \in \R^n \setminus \{ 0 \}$ and the map~$\Psi_H: \R^n \to \R^n$, defined by setting
$$
\Psi_H(\xi) := H(\xi) \nabla H(\xi),
$$
for any~$\xi \in \R^n$, is a global homeomorphism of~$\R^n$, with inverse~$\Psi_{H^*}$.
\end{lemma}

\begin{proof}
Notice that~$B \circ H \in C^2(\R^n \setminus \{ 0 \}) \cap C^0(\R^n)$ and its Hessian is positive definite in~$\R^n \setminus \{ 0 \}$. Hence~$B \circ H$ is strictly convex in the whole of~$\R^n$. Moreover, being~$B'$ positive by~\eqref{Bpos}, the ball~$B_1^H$ is also a sublevel set of~$B \circ H$ and thus strictly convex.

The other claims are valid by virtue of~\cite[Lemma~3.1]{CS09}. Note that~$H$ is assumed to be even in~\cite{CS09}, 
but this assumption is not used in the proof of Lemma~3.1 there. Hence
this result is valid also in our setting. 

Moreover,
$H^*$ is of class~$C^2$ outside of the origin,
since so is
the diffeomorphism~$\Psi_H$.
\end{proof}

%

Next we see that if~$B$ is of the type of the regularized~$p$-Laplacian, i.e. when~(A)\ensuremath{'} holds with~$\bar{\kappa} > 0$, then it is close to being quadratic. In particular, we show that~$B$ can modified far from the origin to make it satisfy~(A)\ensuremath{'} with~$p = 2$. We will need such a trick in Section~\ref{monformsec} in order to overcome a technical difficulty along the proof of Proposition~\ref{uepsexis}.

\begin{lemma} \label{Bcaplem}
Let~$B \in C^2((0, +\infty)) \cap C^1([0, +\infty))$ be a function satisfying both~\eqref{Bpos} and~$B(0) = B'(0) = 0$. Assume in addition that~$B$ satisfies the inequalities displayed in~(A)\ensuremath{'} for some~$p > 1$ and~$\bar{\kappa} > 0$. Let~$M > 0$ be fixed and define
\begin{equation} \label{Bcapdef}
\hat{B}(t) := \begin{cases}
B(t), & \qquad \mbox{if } t \in [0, M), \\
a (t - M)^2 + b (t - M) + c, & \qquad \mbox{if } t \ge M,
\end{cases}
\end{equation}
where~$a = B''(M) / 2$,~$b = B'(M)$ and~$c = B(M)$. Then,~$\hat{B} \in C^2((0, +\infty)) \cap C^1([0, +\infty))$ and it satisfies the inequalities in~(A)\ensuremath{'} with~$p = 2$.
\end{lemma}

\begin{proof}
The function~$\hat{B}$ is of class~$C^2((0, +\infty)) \cap C^1([0, +\infty))$ by construction. Moreover, the estimates concerning~$\hat{B}'$ in~(A)\ensuremath{'} result from the analogous for~$\hat{B}''$ by integration, since~$\hat{B}'(0) = 0$. Thus, we only need to check that there exist~$\hat{\Gamma} \ge \hat{\gamma} > 0$ for which
\begin{equation*}
\hat{\gamma} \le \hat{B}''(t) \le \hat{\Gamma} \mbox{ for any } t > 0.
\end{equation*}
Notice that when~$t \ge M$ this fact is obviously true. On the other hand, if~$t \in (0, M)$ we compute
$$
\hat{B}''(t) = B''(t) \ge \bar{\gamma} (\bar{\kappa} + t)^{p - 2} \ge \bar{\gamma} \min \left\{ \bar{\kappa}^{p - 2}, (\bar{\kappa} + M)^{p - 2} \right\} =: \hat{\gamma},
$$
and
$$
\hat{B}''(t) = B''(t) \le \bar{\Gamma} (\bar{\kappa} + t)^{p - 2} \le \bar{\Gamma} \max \left\{ \bar{\kappa}^{p - 2}, (\bar{\kappa} + M)^{p - 2} \right\} =: \hat{\Gamma}.
$$
This finishes the proof.
\end{proof}

To conclude the section, we comment on the regularity of bounded weak solutions to~\eqref{eleq}. The result is an application of the standard interior degenerate (or non-degenerate) elliptic regularity theory of~\cite{LU68},~\cite{DiB83} and~\cite{T84}. See~\cite[Section~3]{CFV14} for more details.

\begin{proposition} \label{ureg}
Let~$u$ be as in Theorem~\ref{monformthm}. Then,~$u \in C_\loc^{1, \alpha}(\R^n) \cap C^3\left( \{ \nabla u \ne 0 \} \right)$, for some~$\alpha \in (0, 1)$, and~$\nabla u \in L^\infty(\R^n)$.

Moreover, if~$(ii)$ of Theorem~\ref{monformthm} is in force, then~$u$ is of class~$C_\loc^{3, \alpha}(\R^n)$.
\end{proposition}

\section{On the equivalence between assumptions~(A)-(B) and~(A)\ensuremath{'}-(B)\ensuremath{'}} \label{ABsec}

In this second preliminary section we prove the equivalence of the two couples of structural conditions stated in the introduction. We show that both~(A) and~(B) respectively boil down to the simpler and more operational~(A)\ensuremath{'} and~(B)\ensuremath{'}. First, we have

\begin{proposition} \label{(A)char}
Let~$B \in C^2((0, +\infty))$ be a function satisfying~\eqref{Bpos} and~$H \in C^2(\R^n \setminus \{ 0 \})$ be positive homogeneous of degree~$1$, such that~\eqref{Hpos} is true. Then, assumptions~(A) and~(A)\ensuremath{'} are equivalent. Moreover, we may take
\begin{equation} \label{kappabarkappa}
\bar{\kappa} = \kappa,
\end{equation}
and the constants~$\bar{\gamma}, \bar{\Gamma}, \lambda$ and~$\gamma, \Gamma$ to be independent of~$\kappa$.
\end{proposition}

\begin{proof}
First of all, denote with~$C \ge 1$ a constant for which
$$
C^{-1} |\xi| \le H(\xi) \le C |\xi|, \mbox{ } |\nabla H(\xi)| \le C \mbox{ and } \left| \Hess(H) \right| \le C |\xi|^{-1},
$$
hold for any~$\xi \in \R^n \setminus \{ 0 \}$. Then, observe that the ellipticity and growth conditions displayed in~(A) are respectively equivalent to
\begin{align}
\label{ellexp} \left[ B''(H(\xi)) H_i(\xi) H_j(\xi) + B'(H(\xi)) H_{i j}(\xi) \right] \zeta_i \zeta_j & \ge \gamma \left( \kappa + |\xi| \right)^{p - 2} |\zeta|^2, \\
\label{groexp} \sum_{i, j = 1}^n \left| B''(H(\xi)) H_i(\xi) H_j(\xi) + B'(H(\xi)) H_{i j}(\xi) \right| & \le \Gamma \left( \kappa + |\xi| \right)^{p - 2},
\end{align}
for any~$\xi \in \R^n \setminus \{ 0 \}$ and~$\zeta \in \R^n$. 

We start by showing that~(A)\ensuremath{'} implies~(A), in its above mentioned equivalent form. First, we check that~\eqref{groexp} is true. We have
\begin{align*}
 \sum_{i, j = 1}^n \left| B''(H(\xi)) H_i(\xi) H_j(\xi) + B'(H(\xi)) H_{i j}(\xi) \right| & \le \bar{\Gamma} (\bar{\kappa} + H(\xi))^{p - 2} \left[ C^2 + C H(\xi) |\xi|^{-1}  \right] \\
 & \le 2 \bar{\Gamma} C^2 (\bar{\kappa} + c_* |\xi|)^{p - 2} \\
 & = 2 \bar{\Gamma} C^2 c_*^{p - 2} (c_*^{-1} \bar{\kappa} + |\xi|)^{p - 2},
\end{align*}
with
\begin{equation} \label{c*def}
c_* := \begin{cases}
C & \qquad \mbox{if } p \ge 2, \\
1/C & \qquad \mbox{if } 1 < p < 2.
\end{cases}
\end{equation}
The proof of~\eqref{ellexp} is a bit more involved. We write
\begin{equation} \label{zetadec}
\zeta = \alpha \xi + \eta,
\end{equation}
for some~$\alpha \in \R$ and~$\eta \in \nabla H(\xi)^\perp$. We stress that~$\xi$ and~$\nabla H(\xi)^\perp$ span the whole~$\R^n$ in view of~\eqref{i}. Thus, decomposition~\eqref{zetadec} is admissible. We distinguish between the two cases:~$2 |\alpha \xi| \le |\zeta|$ and~$2 |\alpha \xi| > |\zeta|$. In the first situation, we have
$$
|\eta|^2 = |\zeta - \alpha \xi|^2 = |\zeta|^2 - 2 \alpha \langle \zeta, \xi \rangle + \alpha^2 |\xi|^2 \ge (|\zeta| - |\alpha \xi|)^2 \ge \frac{|\zeta|^2}{4}.
$$
Therefore, by applying~\eqref{i},~\eqref{ii} and~\eqref{Huniell2}, we get
\begin{align*}
& \left[ B''(H(\xi)) H_i(\xi) H_j(\xi) + B'(H(\xi)) H_{i j}(\xi) \right] \zeta_i \zeta_j \\
& \qquad = B''(H(\xi)) (H_i(\xi) \zeta_i)^2 + B'(H(\xi)) H_{i j}(\xi) \eta_i \eta_j \ge 0 + \bar{\gamma} (\bar{\kappa} + H(\xi))^{p - 2} H(\xi) \lambda |\xi|^{-1} |\eta|^2 \\
& \qquad \ge 4^{-1} \bar{\gamma} \lambda C^{-1} (\bar{\kappa} + c_*^{-1} |\xi|)^{p - 2} |\zeta|^2 = 4^{-1} \bar{\gamma} \lambda C^{-1} c_*^{2 - p} (c_* \bar{\kappa} + |\xi|)^{p - 2} |\zeta|^2 \\
& \qquad \ge 4^{-1} \bar{\gamma} \lambda C^{-1} c_*^{2 - p} (c_*^{-1} \bar{\kappa} + |\xi|)^{p - 2} |\zeta|^2,
\end{align*}
where in last line we recognized that, for every~$p > 1$,
\begin{equation} \label{CC-1}
(c_* \bar{\kappa} + s)^{p - 2} \ge (c_*^{-1} \bar{\kappa} + s)^{p - 2} \mbox{ for any } s > 0,
\end{equation}
being~$C \ge 1$. On the other hand, if the opposite inequality occurs we deduce that, by~\eqref{i},
$$
\left| \langle \nabla H(\xi), \zeta \rangle \right| = \left| \langle \nabla H(\xi), \alpha \xi + \eta \rangle \right| = |\alpha| H(\xi) \ge \frac{|\alpha| |\xi|}{C} \ge \frac{|\zeta|}{2 C},
$$
so that, we compute
\begin{align*}
& \left[ B''(H(\xi)) H_i(\xi) H_j(\xi) + B'(H(\xi)) H_{i j}(\xi) \right] \zeta_i \zeta_j \\
& \qquad = B''(H(\xi)) (H_i(\xi) \zeta_i)^2 + B'(H(\xi)) H_{i j}(\xi) \eta_i \eta_j \ge \bar{\gamma} (\bar{\kappa} + H(\xi))^{p - 2} (2 C)^{-2} |\zeta|^2 + 0 \\
& \qquad \ge 4^{-1} \bar{\gamma} C^{-2} (\bar{\kappa} + c_*^{-1} |\xi|)^{p - 2} |\zeta|^2 = 4^{-1} \bar{\gamma} C^{-2} c_*^{2 - p} (c_* \bar{\kappa} + |\xi|)^{p - 2} |\zeta|^2 \\
& \qquad \ge 4^{-1} \bar{\gamma} C^{-2} c_*^{2 - p} (c_*^{-1} \bar{\kappa} + |\xi|)^{p - 2} |\zeta|^2,
\end{align*}
and thus the proof of~\eqref{ellexp} is complete.

Now, we focus on the opposite implication, i.e. that~(A) implies~(A)\ensuremath{'}. Let~$t > 0$ and take~$\xi \in \R^n \setminus \{ 0 \}$ such that~$t = H(\xi)$. Plugging~$\zeta = \xi$ in~\eqref{ellexp}, by~\eqref{i} and~\eqref{ii} we obtain
$$
\gamma \left( \kappa + |\xi| \right)^{p - 2} |\xi|^2 \le \left[ B''(t) H_i(\xi) H_j(\xi) + B'(t) H_{i j}(\xi) \right] \xi_i \xi_j = B''(t) H^2(\xi),
$$
and hence that
$$
B''(t) \ge \gamma C^{-2} (\kappa + c_*^{-1} t)^{p - 2} = \gamma C^{-2} c_*^{2 - p} (c_* \kappa + t)^{p - 2}.
$$
On the other hand, the choice~$\zeta \in \nabla H(\xi)^\perp$ in~\eqref{ellexp} leads to
\begin{equation} \label{tech1}
\begin{aligned}
\gamma \left( \kappa + |\xi| \right)^{p - 2} |\zeta|^2 & \le \left[ B''(t) H_i(\xi) H_j(\xi) + B'(t) H_{i j}(\xi) \right] \zeta_i \zeta_j  = B'(t) H_{i j}(\xi) \zeta_i \zeta_j \\
& \le C B'(t) |\xi|^{-1} |\zeta|^2 \le C^2 B'(t) t^{-1} |\zeta|^2.
\end{aligned}
\end{equation}
As before we deduce
$$
B'(t) \ge \gamma C^{-2} c_*^{2 - p} (c_* \kappa + t)^{p - 2} t.
$$
The remaining inequalities involving~$B'$ and~$B''$ in~(A)\ensuremath{'} can be similarly deduced from~\eqref{groexp}. Indeed, notice that~\eqref{i} and~\eqref{ii} respectively yield
\begin{align*}
H_1(e_1) & = \langle \nabla H(e_1), e_1 \rangle = H(e_1), \\
H_{1 1} (e_1) & = \langle \nabla^2 H(e_1) e_1, e_1 \rangle = 0.
\end{align*}
Hence, if we take~$\mu > 0$ such that~$t = H(\mu e_1)$, setting~$\xi = \mu e_1$ in~\eqref{groexp} we get
\begin{align*}
\Gamma \left( \kappa + |\xi| \right)^{p - 2}& \ge \sum_{i, j = 1}^n \left| B''(t) H_i(\mu e_1) H_j(\mu e_1) + B'(t) H_{i j}(\mu e_1) \right| \\
& \ge B''(t) H_1(e_1) H_1(e_1) + B'(t) \mu^{-1} H_{1 1}(e_1) \\
& = B''(t) H^2(e_1).
\end{align*}
Consequently, recalling~\eqref{CC-1} we obtain
$$
B''(t) \le \Gamma C^2 (\kappa + c_* t)^{p - 2} = \Gamma C^2 c_*^{p - 2} (c_*^{-1} \kappa + t)^{p - 2} \le \Gamma C^2 c_*^{p - 2} (c_* \kappa + t)^{p - 2}.
$$
As a byproduct, the previous inequality implies in particular that
$$
B'(1) = \int_0^1 B''(t) \, dt \le \frac{\Gamma C^2 c_*^{p - 2}}{p - 1} (c_* \kappa + 1)^{p - 1}.
$$
Hence, by taking~$t = 1$ in the first line of~\eqref{tech1} we see that~$H$ is uniformly elliptic, with constant
\begin{equation} \label{lambdadef}
\lambda = \frac{(p - 1) c_*^{2(2 - p)} \gamma}{2 C^2 (c_* + 1) \Gamma}.
\end{equation}
Note that we took advantage of the fact that~$\kappa < 1$, along with definition~\eqref{c*def}, to deduce this bound. Finally, the growth condition on~$B'$ can be obtained as follows. Select~$\xi \in \R^n \setminus \{ 0 \}$ in a way that~$e_1 \in \nabla H(\xi)^\perp$ and~$H(\xi) = t$. This can be easily done for instance by taking~$\xi = t \nabla H^*(e_2)$. Indeed, by Lemma~\ref{H*reg}, together with the homogeneity properties of~$H$ and~$\nabla H$, we have
\begin{align*}
0 & = \langle e_2, e_1 \rangle = H(H^*(e_2) \nabla H^*(e_2)) \langle \nabla H(H^*(e_2) \nabla H^*(e_2)), e_1 \rangle \\
& = H^*(e_2) H(\nabla H^*(e_2)) \langle \nabla H(\nabla H^*(e_2)), e_1 \rangle = H^*(e_2) \langle \nabla H(\nabla H^*(e_2)), e_1 \rangle.
\end{align*}
Such a choice implies that
$$
\langle \nabla H(\xi), e_1 \rangle = \langle \nabla H( \nabla H^*(e_2)), e_1 \rangle = 0.
$$
Moreover, it is easy to see that~$H(\xi) = t$. From~\eqref{groexp} we may then compute
\begin{align*}
\Gamma \left( \kappa + |\xi| \right)^{p - 2}& \ge \sum_{i, j = 1}^n \left| B''(t) H_i(\xi) H_j(\xi) + B'(t) H_{i j}(\xi) \right| \\
& \ge B''(t) H_1(\xi) H_1(\xi) + B'(t) H_{1 1}(\xi) \\
& = B'(t) H_{1 1}(\xi) \ge B'(t) \lambda |\xi|^{-1},
\end{align*}
from which we get, as before,
$$
B'(t) \le \Gamma \lambda^{-1} C c_*^{p - 2} (c_* \kappa + t)^{p - 2} t,
$$
with~$\lambda$ as in~\eqref{lambdadef}. This concludes the proof of the second part of our claim.

The fact that we may assume~\eqref{kappabarkappa} to hold - up to relabeling the constants~$\gamma, \Gamma$ or~$\bar{\gamma}, \bar{\Gamma}$ in dependence of~$C$ - is a consequence of the inequalities
$$
(\kappa + t)^{p - 2} \le (c_* \kappa + t)^{p - 2} \le c_*^{p - 2} (\kappa + t)^{p - 2},
$$
and
\begin{equation*}
c_*^{2 - p} (\bar{\kappa} + |\xi|)^{p - 2} \le (c_*^{- 1}\bar{\kappa} + |\xi|)^{p - 2} \le (\bar{\kappa} + |\xi|)^{p - 2}. \qedhere
\end{equation*}
\end{proof}

On the other hand, the characterization of~(B) in terms of~(B)\ensuremath{'} is the content of the following

\begin{proposition} \label{(B)char}
Let~$B \in C^3((0, +\infty)) \cap C^1([0, +\infty))$ be a function satisfying~\eqref{Bpos} along with~$B(0) = B'(0) = 0$ and~$H \in C^3(\R^n \setminus \{ 0 \})$ be positive homogeneous of degree~$1$, such that~\eqref{Hpos} is true. Then, hypotheses~(B) and~(B)\ensuremath{'} are equivalent.
\end{proposition}

\begin{proof}
We begin by showing that~(B)\ensuremath{'} implies~(B). First, we deal with the regularity of the composition~$B \circ H$. By the general assumptions on~$B$ and~$H$, it is clear that~$B \circ H \in C_\loc^{3, \beta}(\R^n \setminus \{ 0 \}) \cap C^1(\R^n)$. Thus, we only need to check the second and third derivatives of~$B \circ H$ at the origin. For any~$e \in S^{n - 1}$ and~$t > 0$, by the homogeneity of~$H$ we have
\begin{align*}
(B \circ H)_{i j}(t e) & = B''(H(t e)) H_i(t e) H_j(t e) + B'(H(t e)) H_{i j}(t e) \\
& = B''(t H(e)) H_i(e) H_j(e) + \frac{B'(t H(e))}{t H(e)} H(e) H_{i j}(e).
\end{align*}
Hence, taking the limit as~$t \rightarrow 0^+$
\begin{equation} \label{BH20}
\lim_{t \rightarrow 0^+} (B \circ H)_{i j}(t e) = B''(0) \left[ H_i(e) H_j(e) + H(e) H_{i j}(e) \right].
\end{equation}
Now, observe that, being~$H$ of the special form~\eqref{HM}, we may explicitly compute
\begin{equation} \label{Mij}
M_{i j} = \partial_{i j} \left( \frac{H^2}{2} \right) (\xi) = H_i(\xi) H_j(\xi) + H(\xi) H_{i j}(\xi),
\end{equation}
for any~$\xi \in \R^n$. As a consequence of~\eqref{Mij}, the right hand side of~\eqref{BH20} does not depend on~$e \in S^{n - 1}$ and so
$$
(B \circ H)_{i j}(0) = B''(0) M_{i j}.
$$
Now we focus on the third derivative. First, by differentiating~\eqref{Mij} we deduce the identity
$$
H_i(\xi) H_{j k}(\xi) + H_j(\xi) H_{i k}(\xi) + H_k(\xi) H_{i j}(\xi) = - H(\xi) H_{i j k}(\xi).
$$
With this in hand we compute
\begin{equation} \label{BH3}
\begin{aligned}
(B \circ H)_{i j k}(t e) & = B'''(H(t e))H_i(t e) H_j(t e) H_k(t e) \\
& \quad + B''(H(t e)) \left[ H_i(t e) H_{j k}(t e) + H_j(t e) H_{i k}(t e) + H_k(t e) H_{i j}(t e) \right] \\
& \quad + B'(H(t e)) H_{i j k}(t e) \\
& = B'''(t H(e)) H_i(e) H_j(e) H_k(e) \\
& \quad +\frac{1}{t H(e)} \left[ \frac{B'(t H(e))}{t H(e)} - B''(t H(e)) \right] H^2(e) H_{i j k}(e).
\end{aligned}
\end{equation}
Now, we claim that
\begin{equation} \label{limit0}
\lim_{s \rightarrow 0^+} \frac{1}{s} \left[ \frac{B'(s)}{s} - B''(s) \right] = 0.
\end{equation}
Indeed, since~$B'(0) = B'''(0) = 0$, the Taylor expansions of~$B'$ and~$B''$ are
$$
B'(s) = B''(0) s + o(s^2) \quad \mbox{and} \quad B''(s) = B''(0) + o(s),
$$
as~$s \rightarrow 0^+$. Therefore
$$
\frac{B'(s)}{s} - B''(s) = \frac{B''(0) s}{s} - B''(0) + o(s) = o(s),
$$
and~\eqref{limit0} follows. Thus, letting~$t \rightarrow 0^+$ in~\eqref{BH3}, we get
$$
\lim_{t \rightarrow 0^+}(B \circ H)_{i j k}(t e) = B'''(0) H_i(e) H_j(e) H_k(e) + 0 \cdot H^2(e) H_{i j k}(e) = 0,
$$
for any~$e \in S^{n - 1}$. We may thence conclude that~$B \circ H \in C_\loc^{3, \beta}(\R^n)$. Finally, we prove that~$\Hess \left( B \circ H \right)$ is uniformly elliptic on compact subsets, as required in~(B). Let
\begin{equation} \label{Cmax}
C := \max_{\xi \in S^{n - 1}} H(\xi).
\end{equation}
By~\eqref{Bpos} and~\eqref{B''0pos}, for any~$K > 0$, there exists~$\bar{\gamma} > 0$ such that
\begin{equation} \label{B''K}
B''(t) \ge \bar{\gamma},
\end{equation}
for any~$t \in [0, C K]$. Since~$B'(0) = 0$, we also infer that
\begin{equation} \label{B'K}
B'(t) = \int_0^t B''(s) \, ds \ge \bar{\gamma} t,
\end{equation}
for any~$t \in [0, C K]$. Let~$\xi, \eta \in \R^n$, with~$|\xi| \le K$. Observe that, by~\eqref{Cmax}, it holds
$$
H(\xi) \le |\xi| H \left( \frac{\xi}{|\xi|} \right) \le C K.
$$
Then, by~\eqref{B''K},~\eqref{B'K} and~\eqref{Mij},
\begin{align*}
(B \circ H)_{i j}(\xi) \eta_i \eta_j & = \left[ B''(H(\xi)) H_i(\xi) H_j(\xi) + B'(H(\xi)) H_{i j}(\xi) \right] \eta_i \eta_j \\
& \ge \bar{\gamma} \left[ H_i(\xi) H_j(\xi) + H(\xi) H_{i j}(\xi) \right] \eta_i \eta_j \\
& = \bar{\gamma} M_{i j} \eta_i \eta_j,
\end{align*}
and the result follows from the positive definiteness of~$M$.

Now, we turn to the converse implication, i.e. that~(B) implies~(B)\ensuremath{'}. First, we observe that~$H$ needs to be of the type~\eqref{HM}, in view of~\cite[Appendix~B]{CFV14}. Then, we address the regularity of~$B$. Being~$H$ even and using~\eqref{i}, we have
\begin{align*}
B''(0) & = \lim_{t \rightarrow 0^+} \frac{B'(t)}{t} = \lim_{s \rightarrow 0} \frac{B'(H(s e_1))}{H(s e_1)} = \lim_{s \rightarrow 0} \frac{(B \circ H)_1(s e_1)}{H(s e_1) H_1(s e_1)} \\
& = \lim_{s \rightarrow 0} \frac{s}{s H(e_1) H_1(e_1)} \frac{(B \circ H)_1(s e_1) - (B \circ H)_1(0)}{s} \\
& = H^{-2}(e_1) (B \circ H)_{1 1}(0),
\end{align*}
so that~$B \in C^2([0, +\infty))$. Moreover,~$B''(0) > 0$, as can be seen by testing with~$\xi = 0, \zeta = e_1$ the ellipticity condition of~(B). On the other hand, by~\eqref{i} and~\eqref{ii} we compute
\begin{align*}
(B \circ H)_{1 1 1}(0) & = \lim_{t \rightarrow 0} \frac{(B \circ H)_{1 1}(t e_1) - (B \circ H)_{1 1}(0)}{t} \\
& = \lim_{t \rightarrow 0} \frac{B''(H(t e_1)) H_1(t e_1) H_1(t e_1) + B'(H(t e_1)) H_{1 1}(t e_1) - B''(0) H^2(e_1)}{t} \\
& = \pm H^3(e_1) \lim_{t \rightarrow 0^\pm} \frac{B''(|t| H(e_1)) - B''(0)}{|t| H(e_1)} \\
& = \pm H^3(e_1) \lim_{s \rightarrow 0^+} \frac{B''(s) - B''(0)}{s}.
\end{align*}
Since the left hand side exists finite, the same should be true for the right one, too. Thus, we obtain that~$B \in C^3([0, +\infty))$ with~$B'''(0) = 0$. This concludes the proof, as the local H\"olderianity of~$B$ up to~$0$ may be easily deduced from that of~$B \circ H$.
\end{proof}

We remark that Theorem~\ref{T-RES} now follows
easily from Propositions~\ref{(A)char}
and~\ref{(B)char}.

\section{The monotonicity formula} \label{monformsec}
In this section we prove Theorem~\ref{monformthm}. Our argument is similar to that presented in~\cite{M89} and~\cite[Theorem 1.4]{CGS94}. Yet, we develop several technical adjustments in order to cope with the difficulties arising in the anisotropic setting.
In particular, in the classical, isotropic setting, the monotonicity
formulae implicitly rely on some Euclidean geometric features,
such as that a point on the unit sphere
coincides with the normal of the sphere at that point,
as well as the one of the dual sphere (that in the
isotropic setting coincides with the original one).
These Euclidean geometric properties are lost in our
case, therefore we need some more refined
geometrical and analytical studies.

The strategy we adopt to show the monotonicity of~$\E$ basically relies on taking its derivative and then checking that it is non-negative. To complete this task, however, we make some integral manipulation involving the Hessian of~$u$. Hence, we need~$u$ to be twice differentiable, at least in the weak sense.

If~$(ii)$ is assumed to hold, this is not an issue, since~$u$ is~$C^3$ (see Proposition~\ref{ureg}). Therefore, we only focus on case~$(i)$. In this framework the solution~$u$ is, in general, no more than~$C_\loc^{1, \alpha}$. To circumvent this lack of regularity, we introduce a sequence of approximating problems and perform the computation on their solutions. Passing to the limit, we then recover the result for~$u$. If one is interested in the proof under hypothesis~$(ii)$, he should simply ignore the perturbation argument and directly work with~$u$.

Prior to the proper proof of Theorem~\ref{monformthm}, we present some preparatory results about the above mentioned approximation technique. In every statement the functions~$B$,~$H$ and~$u$ are assumed to be satisfy assumption~$(i)$.

Let~$\varepsilon \in (0, 1)$ and consider the function~$B_\varepsilon$ defined by
\begin{equation} \label{Bepsdef}
B_\varepsilon(t) := B \left( \sqrt{\varepsilon^2 + t^2} \right) - B(\varepsilon),
\end{equation}
for any~$t > 0$.

First, we present a result which addresses the regularity and growth properties of~$B_\varepsilon$.

\begin{lemma} \label{BepsHlem}
The function~$B_\varepsilon$ is of class~$C^2([0, +\infty))$ and it satisfies~$B_\varepsilon(0) = B_\varepsilon'(0) = 0$ and~\eqref{Bpos}. Moreover
\begin{equation} \label{Bepsellgro}
\begin{aligned}
c_p \bar{\gamma} \left( \bar{\kappa} + \varepsilon + t \right)^{p - 2} t \le & B_\varepsilon'(t) \le C_p \bar{\Gamma} \left( \bar{\kappa} + \varepsilon + t \right)^{p - 2} t, \\
c_p \bar{\gamma} \left( \bar{\kappa} + \varepsilon + t \right)^{p - 2} \le & B_\varepsilon''(t) \le C_p \bar{\Gamma} \left( \bar{\kappa} + \varepsilon + t \right)^{p - 2},
\end{aligned}
\end{equation}
for any~$t > 0$, where~$\bar{\gamma}, \bar{\Gamma}$ are as in~(A)\ensuremath{'} and
$$
c_p := \min \left\{ 1, 2^{\frac{2 - p}{2}} \right\}, \, C_p := \max \left\{ 1, 2^{\frac{2 - p}{2}} \right\}.
$$
In addition, the composition~$B_\varepsilon \circ H$ is of class~$C^{1, 1}_\loc(\R^n)$ and it holds, for any~$\xi \in \R^n$,
\begin{equation} \label{BepsHell}
(B_\varepsilon \circ H)(\xi) \ge \frac{\gamma}{2 (p - 1) p} |\xi|^p - c_\star,
\end{equation}
where~$\gamma$ is as in~(A) and~$c_\star$ is a non-negative constant independent of~$\varepsilon$.
\end{lemma}

\begin{proof}
It is immediate to check from definition~\eqref{Bepsdef} that~$B_\varepsilon \in C^2([0, +\infty))$. For any~$t > 0$, we compute
\begin{align*}
B_\varepsilon'(t) & = B' \left( \sqrt{\varepsilon^2 + t^2} \right) \frac{t}{\sqrt{\varepsilon^2 + t^2}}, \\
B_\varepsilon''(t) & = B'' \left( \sqrt{\varepsilon^2 + t^2} \right) \frac{t^2}{\varepsilon^2 + t^2} + B' \left( \sqrt{\varepsilon^2 + t^2} \right) \frac{\varepsilon^2}{\left( \varepsilon^2 + t^2 \right)^{3/2}}.
\end{align*}
Thus, inequalities~\eqref{Bpos} are valid and~$B_\varepsilon(0) = B_\varepsilon'(0) = 0$. Furthermore, formulae~\eqref{Bepsellgro} can be recovered from the ellipticity and growth conditions of~(A)\ensuremath{'} which~$B$ satisfies.

Then, we address the composition~$B_\varepsilon \circ H$. Notice that we already know that it is of class~$C^1$ on the whole~$\R^n$, by virtue of Lemma~\ref{derBH0}, and~$C^2$ outside of the origin, by definition. Thus we only need to check that its gradient is Lipschitz in a neighbourhood of the origin. By using~\eqref{Bepsellgro}, for any~$0 < |\xi| \le 1$ we get
$$
\frac{\left| \partial_i (B_\varepsilon \circ H)(\xi) \right|}{|\xi|} = \frac{\left| B_\varepsilon'(H(\xi)) H_i(\xi) \right|}{|\xi|} \le C_p \bar{\Gamma} (\bar{\kappa} + \varepsilon + H(\xi))^{p - 2} H_i(\xi) \frac{H(\xi)}{|\xi|} \le c,
$$
for some positive~$c$.

Finally, we establish~\eqref{BepsHell}. As a preliminary observation, we stress that the Hessian of~$B_\varepsilon \circ H$ satisfies~(A) with~$\kappa = \bar{\kappa} + \varepsilon$. This can be seen as a consequence of~\eqref{Bepsellgro}, the uniform ellipticity of~$H$ and Proposition~\ref{(A)char} (recall in particular relation~\eqref{kappabarkappa}). We consider separately the two possibilities~$p \ge 2$ and~$1 < p < 2$. In the first case, we simply compute
\begin{align*}
(B_\varepsilon \circ H)(\xi) & = \int_0^1 \int_0^t (B_\varepsilon \circ H)_{i j}(s \xi) \xi_i \xi_j \, ds dt \ge \gamma \int_0^1 \int_0^t (\kappa + s |\xi|)^{p - 2} |\xi|^2 \, ds dt \\
& \ge \gamma |\xi|^p \int_0^1 \int_0^t s^{p - 2} \, ds dt = \frac{\gamma}{(p - 1) p} |\xi|^p.
\end{align*}
If, on the other hand,~$1 < p < 2$, we have
\begin{equation} \label{techBepsH}
\begin{aligned}
(B_\varepsilon \circ H)(\xi) & = \int_0^1 \int_0^t (B_\varepsilon \circ H)_{i j}(s \xi) \xi_i \xi_j \, ds dt \ge \gamma \int_0^1 \int_0^t (\kappa + s |\xi|)^{p - 2} |\xi|^2 \, ds dt \\
& = \frac{\gamma}{p - 1} \left[ \frac{(\kappa + |\xi|)^p - \kappa^p}{p} - \kappa^{p - 1} |\xi| \right] \ge \frac{\gamma}{p - 1} \left[ \frac{|\xi|^p - \kappa^p}{p} - \kappa^{p - 1} |\xi| \right].
\end{aligned}
\end{equation}
Notice that, by Young's inequality, we estimate
$$
\kappa^{p - 1} |\xi| \le \frac{|\xi|^p}{2 p} + \frac{p - 1}{p} 2^{1/(p - 1)} \kappa^p.
$$
Plugging this into~\eqref{techBepsH} finally leads to the desired
\begin{align*}
(B_\varepsilon \circ H)(\xi) & \ge \frac{\gamma}{2 (p - 1) p} |\xi|^p - \frac{\gamma}{(p - 1) p} \left( 1 + (p - 1) 2^{1 / (p - 1)} \right) \kappa^p \\
& \ge \frac{\gamma}{2 (p - 1) p} |\xi|^p - \frac{\gamma}{(p - 1) p} \left( 1 + (p - 1) 2^{1 / (p - 1)} \right) (\bar{\kappa} + 1)^p.
\end{align*}
Hence,~\eqref{BepsHell} holds in both cases and the proof of the lemma is complete.
\end{proof}

In the following lemma we compare~$B_\varepsilon$ to~$B$. We study their modulus of continuity and discuss some uniform convergence properties.

\begin{lemma} \label{BepsBlem}
Introduce, for~$t \ge 0$, the functions~$\beta(t) := B'(t) t$,~$\beta_\varepsilon(t) := B_\varepsilon'(t) t$.

Then, the Lipschitz norms of both~$B_\varepsilon$ and~$\beta_\varepsilon$ on compact sets of~$[0, +\infty)$ are bounded by a constant independent of~$\varepsilon$. More explicitly, for any~$M \ge 1$ we estimate
\begin{equation} \label{Bbetaepslip}
\begin{aligned}
\| B_\varepsilon \|_{C^{0, 1}([0, M])} & \le \| B \|_{C^{0, 1}([0, 2 M])}, \\
\| \beta_\varepsilon \|_{C^{0, 1}([0, M])} & \le 2 \| B' \|_{C^0([0, 2 M])} + \| \beta \|_{C^{0, 1}([0, 2 M])}.
\end{aligned}
\end{equation}
Moreover,~$B_\varepsilon \rightarrow B$ and~$\beta_\varepsilon \rightarrow \beta$ uniformly on compact sets of~$[0, +\infty)$. Quantitatively, we have
\begin{equation} \label{Bbetaepsconv}
\begin{aligned}
\| B_\varepsilon - B \|_{C^0([0, M])} & \le 2 \| B' \|_{C^0([0, 2M])} \varepsilon, \\
\| \beta_\varepsilon - \beta \|_{C^0([0, M])} & \le \left( \| B' \|_{C^0([0, 2M])} + \| \beta \|_{C^{0, 1}([0, 2M])} \right) \varepsilon.
\end{aligned}
\end{equation}
\end{lemma}

\begin{proof}
First of all, we stress that, while~$\beta_\varepsilon \in C^1([0, +\infty))$ in view of Lemma~\ref{BepsHlem}, the same is true also for~$\beta$, as one can easily deduce from hypothesis~(A)\ensuremath{'}.

We begin to establish~\eqref{Bbetaepslip}. It is easy to see that the~$C^0$ norms of~$B_\varepsilon$ and~$\beta_\varepsilon$ are bounded by those of~$B$ and~$\beta$ respectively. Thus, we may concentrate on the estimates of their Lipschitz seminorms. Let~$M \ge 1$ and~$0 \le s, t \le M$. We have
\begin{align*}
|B_\varepsilon(t) - B_\varepsilon(s)| & = \left| B \left( \sqrt{\varepsilon^2 + t^2} \right) - B \left( \sqrt{\varepsilon^2 + s^2} \right) \right| \\
& \le \| B \|_{C^{0, 1}([0, 2M])} \left| \sqrt{\varepsilon^2 + t^2} - \sqrt{\varepsilon^2 + s^2} \right| \\
& \le \| B \|_{C^{0, 1}([0, 2M])} |t - s|,
\end{align*}
so that the first relation in~\eqref{Bbetaepslip} is proved. The second inequality needs a little more care. Assuming without loss of generality~$s \le t$, we compute
\begin{align*}
|\beta_\varepsilon(t) - \beta_\varepsilon(s)| & = \left| B' \left( \sqrt{\varepsilon^2 + t^2} \right) \frac{t^2}{\sqrt{\varepsilon^2 + t^2}} - B' \left( \sqrt{\varepsilon^2 + s^2} \right) \frac{s^2}{\sqrt{\varepsilon^2 + s^2}} \right| \\
& \le B' \left( \sqrt{\varepsilon^2 + t^2} \right) \sqrt{\varepsilon^2 + t^2} \left| \frac{t^2}{\varepsilon^2 + t^2} - \frac{s^2}{\varepsilon^2 + s^2} \right| \\
& \quad + \frac{s^2}{\varepsilon^2 + s^2} \left| \beta \left( \sqrt{\varepsilon^2 + t^2} \right) - \beta \left( \sqrt{\varepsilon^2 + s^2} \right) \right| \\
& \le \| B' \|_{C^0([0, 2 M])} \frac{|t^2 - s^2|}{\sqrt{\varepsilon^2 + t^2}} + \| \beta \|_{C^{0, 1}([0, 2M])} \left| \sqrt{\varepsilon^2 + t^2} - \sqrt{\varepsilon^2 + s^2} \right| \\
& \le \left( 2 \| B' \|_{C^0([0, 2 M])} + \| \beta \|_{C^{0, 1}([0, 2M])} \right) |t - s|.
\end{align*}

Estimates~\eqref{Bbetaepsconv} are proved in a similar fashion. Indeed, for any~$0 \le t \le M$,
\begin{align*}
\left| B_\varepsilon(t) - B(t) \right| & = \left| B \left( \sqrt{\varepsilon^2 + t^2} \right) - B(\varepsilon) - B(t) \right| \\
& \le \| B \|_{C^{0, 1}([0, 2 M])} \left( \left| \sqrt{\varepsilon^2 + t^2} - t \right| + \varepsilon \right) \\
& \le 2 \| B \|_{C^{0, 1}([0, 2 M])} \varepsilon,
\end{align*}
and
\begin{align*}
\left| \beta_\varepsilon(t) - \beta(t) \right| & \le B' \left( \sqrt{\varepsilon^2 + t^2} \right) \left| \frac{t^2}{\sqrt{\varepsilon^2 + t^2}} - \sqrt{\varepsilon^2 + t^2} \right| + \left| \beta \left( \sqrt{\varepsilon^2 + t^2} \right) - \beta(t) \right| \\
& \le \left( \| B' \|_{C^0([0, 2 M])} + \| \beta \|_{C^{0, 1}([0, 2 M])} \right) \varepsilon.
\end{align*}
Thus, the proof is complete.
\end{proof}

Next is the key proposition of the approximation argument. Basically, we consider some perturbed problems driven by~$B_\varepsilon$. We prove that their solutions are~$H^2$ regular and that they converge to~$u$.

\begin{proposition} \label{uepsexis}
Let~$\Omega$ be a bounded open set of~$\R^n$ with~$C^{1, \alpha}$ boundary. The problem
\begin{equation} \label{uepsprob}
\begin{cases}
\dive \big( B_\varepsilon'(H(\nabla u^\varepsilon)) \nabla H(\nabla u^\varepsilon) \big) + F'(u) = 0, & \qquad \mbox{in } \Omega, \\
u^\varepsilon = u, & \qquad \mbox{on } \partial \Omega,
\end{cases}
\end{equation}
admits a strong solution~$u^\varepsilon \in C^{1, \alpha'}(\overline{\Omega}) \cap H^2(\Omega)$, for some~$\alpha' \in (0, 1]$ independent of~$\varepsilon$. Furthermore,~$u^\varepsilon$ converges to~$u$ in~$C^1(\overline{\Omega})$, as~$\varepsilon \rightarrow 0^+$.
\end{proposition}

\begin{proof}
By using standard methods - see, for instance,~\cite[Theorem~3.30]{D07} - we know that the functional
$$
\F_\varepsilon(v) := \int_{\Omega} B_\varepsilon(H(\nabla v(x))) - F'(u(x)) v(x) \, dx,
$$
admits the existence of a minimizer~$u^\varepsilon \in W^{1, p}(\Omega)$, with~$u^\varepsilon - u \in W_0^{1, p}(\Omega)$. Note that~$\F_\varepsilon$ is coercive, thanks to~\eqref{BepsHell}, the continuity of~$F'$ and the boundedness of~$u$. Clearly,~$u^\varepsilon$ satisfies~\eqref{uepsprob} in the weak sense.

In view of~\eqref{BepsHell}, we see that the minimizer~$u^\varepsilon$ is bounded in~$\Omega$ (use e.g.~\cite[Theorems~6.1-6.2]{S63} or~\cite[Theorem 3.2, p.~328]{LU68}). Moreover, the~$L^\infty$ norm of~$u^\varepsilon$ is uniform in~$\varepsilon$.

With this in hand, we can now verify that~$u^\varepsilon \in C^{1, \alpha'}$.
For this, we notice that Lemma~\ref{BepsHlem} and Proposition~\ref{(A)char} ensure that hypothesis~(A) is verified by~$B_\varepsilon \circ H$. Hence, by the uniform~$L^\infty$ estimates, we may appeal to~\cite[Theorem~1]{L88} to deduce that~$u^\varepsilon \in C^{1, \alpha'}(\overline{\Omega})$, for some~$\alpha' \in (0, 1]$. Notice that~$\alpha'$ is independent of~$\varepsilon$ and~$\| u^\varepsilon \|_{C^{1, \alpha'}(\overline{\Omega})}$ is uniformly bounded in~$\varepsilon$.

Consequently, by Arzelà-Ascoli Theorem, the sequence $\{ u^\varepsilon \}$ converges in $C^1(\overline{\Omega})$ to a function $v$, as $\varepsilon \rightarrow 0^+$. With the aid of Lemma~\ref{BepsBlem}, we see that~$v$ is the unique solution of
$$
\begin{cases}
\dive \left( B'(H(\nabla v)) \nabla H(\nabla v) \right) + F'(u) = 0, & \qquad \mbox{in } \Omega, \\
v = u, & \qquad \mbox{on } \partial \Omega.
\end{cases}
$$
Therefore,~$v = u$ in the whole~$\overline{\Omega}$.

Now we prove the~$H^2$ regularity of~$u^\varepsilon$. To this aim we employ~\cite[Proposition~1]{T84}. Notice that we need to check the validity of condition~(2.4) there, in order to apply such result. If~$p \ge 2$ it is an immediate consequence of the fact that~$B_\varepsilon \circ H$ satisfies~(A). Indeed, for any~$\eta \in \R^n \setminus \{ 0 \}$,~$\zeta \in \R^n$, we deduce that
$$
\left[ \Hess \,(B_\varepsilon \circ H)(\xi) \right]_{i j} \zeta_i \zeta_j \ge \gamma {(\bar{\kappa} + \varepsilon + |\xi|)}^{p - 2} {|\zeta|}^2 \ge \tilde{\gamma} {|\zeta|}^2,
$$
for some~$\tilde{\gamma} > 0$. In case~$1 < p < 2$, we set~$M := \| \nabla u^\varepsilon \|_{L^\infty(\overline{\Omega})}$ and modify~$B_\varepsilon$ accordingly to Lemma~\ref{Bcaplem}. The new function~$\hat{B}_\varepsilon$ obtained this way satisfies assumption~(A)\ensuremath{'}, and thus~(A), with~$p = 2$. Moreover,~$u^\varepsilon$ is still a weak solution to~\eqref{uepsprob} with~$B_\varepsilon$ replaced by~$\hat{B}_\varepsilon$. This is enough to conclude that~$u^\varepsilon \in H^2(\Omega)$ also when~$1 < p < 2$.

From the additional Sobolev regularity we deduce that~$u^\varepsilon$ is actually a strong solution of~\eqref{uepsprob}. Indeed, it is sufficient to observe that, for any~$i = 1, \ldots, n$,
$$
B_\varepsilon'(H(\nabla u^\varepsilon)) H_i(\nabla u^\varepsilon) = \left( B_\varepsilon \circ H \right)_i (\nabla u^\varepsilon) \in H^1(\Omega),
$$
being~$(B_\varepsilon \circ H)_i$ locally uniformly Lipschitz, by Lemma~\ref{BepsHlem}. 
\end{proof}

After all these preliminary results, we may finally head to the

\begin{proof}[Proof of Theorem~\ref{monformthm}]
First, using the coarea formula we compute
$$
\E'(R) = \frac{1 - n}{R} \E(R) + \frac{1}{R^{n - 1}} \int_{\partial W_R} \left[ B(H(\nabla u)) + G(u) \right] |\nabla H^*|^{-1} \, d\H^{n - 1}.
$$
Then, notice that the exterior unit normal vector to~$\partial W_R$ at~$x \in \partial W_R$ is given by
\begin{equation} \label{nu}
\nu(x) = \frac{\nabla H^*(x)}{|\nabla H^*(x)|}.
\end{equation}
Thus, by the homogeneity of~$H$ and the second identity in~\eqref{CS2} we have
$$
H(\nu(x)) = |\nabla H^*(x)|^{-1} H(\nabla H^*(x)) = |\nabla H^*(x)|^{-1}.
$$
As a consequence, the derivative of~$\E$ at~$R$ becomes
\begin{equation} \label{montech1}
\E'(R) = \frac{1 - n}{R} \E(R) + \frac{1}{R^{n - 1}} \int_{\partial W_R} \left[ B(H(\nabla u)) + G(u) \right] H(\nu) \, d\H^{n - 1}.
\end{equation}

For any~$\varepsilon \in (0, 1)$, let now~$u^\varepsilon \in C^{1, \alpha'}(\overline{W_R}) \cap H^2(W_R)$ be the strong solutions of~\eqref{uepsprob}, with~$\Omega = W_R$. Notice that~$\partial W_R$ is of class~$C^2$ in view of Lemma~\ref{H*reg}. Hence, we are allowed to apply Proposition~\ref{uepsexis} to obtain such a~$u^\varepsilon$. By the results of Proposition~\ref{uepsexis} and Lemma~\ref{BepsBlem}, along with the~$C^2$ regularity of~$G$, it is immediate to check that
\begin{equation} \label{BHuconv}
\begin{aligned}
B_\varepsilon(H(\nabla u^\varepsilon)) & \longrightarrow B(H(\nabla u)), \\
B_\varepsilon'(H(\nabla u^\varepsilon)) H(\nabla u^\varepsilon) & \longrightarrow B'(H(\nabla u)) H(\nabla u), \\
G(u^\varepsilon) \longrightarrow G(u) \, \, & \, \mbox{and} \, \, \, F'(u^\varepsilon) \longrightarrow F'(u),
\end{aligned}
\end{equation}
uniformly on~$\overline{W_R}$.

In view of Lemma~\ref{H*reg} the function~$H \nabla H$ is bijective and its inverse is given by~$H^* \nabla H^*$. Hence, exploiting the homogeneity properties of~$H$ and~$\nabla H$ together with~\eqref{CS2}, it follows that the identity 
\begin{align*}
x & = H(H^*(x) \nabla H^*(x)) \nabla H(H^*(x) \nabla H^*(x)) = H^*(x) H(\nabla H^*(x)) \nabla H(\nabla H^*(x)) \\
& = H^*(x) \nabla H(\nabla H^*(x)),
\end{align*}
is true for any~$x \in \R^n \setminus \{ 0 \}$. Consequently, using~\eqref{i},~\eqref{nu}, the homogeneity of~$\nabla H$, the definition of~$\partial W_R$ and the divergence theorem, we compute
\begin{align*}
& \int_{\partial W_R} B_\varepsilon(H(\nabla u^\varepsilon)) H(\nu) \, d\H^{n - 1} = \frac{1}{R} \int_{\partial W_R} B_\varepsilon(H(\nabla u^\varepsilon)) H^* \langle \nabla H(\nu), \nu \rangle \, d\H^{n - 1} \\
& \qquad = \frac{1}{R} \int_{W_R} \dive \big( B_\varepsilon(H(\nabla u^\varepsilon)) H^* \nabla H(\nabla H^*) \big) \, dx = \frac{1}{R} \int_{W_R} \dive \big( B_\varepsilon(H(\nabla u^\varepsilon)) x \big) \, dx \\
& \qquad = \frac{1}{R} \int_{W_R} B_\varepsilon'(H(\nabla u^\varepsilon)) H_j(\nabla u^\varepsilon) u^\varepsilon_{i j} x_i \, dx + \frac{n}{R} \int_{W_R} B_\varepsilon(H(\nabla u^\varepsilon)) \, dx.
\end{align*}
With a completely analogous argument we also deduce that
$$
\int_{\partial W_R} G(u^\varepsilon) H(\nu) \, d\H^{n - 1} = - \frac{1}{R} \int_{W_R} F'(u^\varepsilon) u^\varepsilon_i x_i \, dx + \frac{n}{R} \int_{W_R} G(u^\varepsilon) \, dx.
$$
Putting these last two identities together we obtain
\begin{equation} \label{montech34}
\int_{\partial W_R} \left[ B_\varepsilon(H(\nabla u^\varepsilon)) + G(u^\varepsilon) \right] H(\nu) \, d\H^{n - 1} = \frac{n}{R} \int_{W_R} B_\varepsilon(H(\nabla u^\varepsilon)) + G(u^\varepsilon) \, dx + \frac{I_\varepsilon}{R},
\end{equation}
where
$$
I_\varepsilon := \int_{W_R} \left[ B_\varepsilon'(H(\nabla u^\varepsilon)) H_j(\nabla u^\varepsilon) u^\varepsilon_{i j} - F'(u^\varepsilon) u^\varepsilon_i \right] x_i \, dx.
$$
Recalling that~$u^\varepsilon$ is a strong solution of~\eqref{uepsprob}, we compute
\begin{align*}
I_\varepsilon & = \int_{W_R} \left[ \big(  B_\varepsilon'(H(\nabla u^\varepsilon)) H_j(\nabla u^\varepsilon) u^\varepsilon_i \big)_j - \left( B_\varepsilon'(H(\nabla u^\varepsilon)) H_j(\nabla u^\varepsilon) \right)_j u^\varepsilon_i - F'(u^\varepsilon) u^\varepsilon_i \right] x_i \, dx \\
& = \int_{W_R} \big(  B_\varepsilon'(H(\nabla u^\varepsilon)) H_j(\nabla u^\varepsilon) u^\varepsilon_i \big)_j x_i \, dx + \int_{W_R} \left[ F'(u) - F'(u^\varepsilon) \right] u^\varepsilon_i x_i \, dx.
\end{align*}
By the divergence theorem, formulae~\eqref{i},~\eqref{nu} and condition~\eqref{FKweak} we find
\begin{equation} \label{montech5}
\begin{aligned}
I_\varepsilon & = \int_{W_R} \left( B_\varepsilon'(H(\nabla u^\varepsilon)) H_j(\nabla u^\varepsilon) u^\varepsilon_i x_i \right)_j \, dx - \int_{W_R} B_\varepsilon'(H(\nabla u^\varepsilon)) H_j(\nabla u^\varepsilon) u^\varepsilon_i \delta_{i j} \, dx \\
& \quad + \int_{W_R} \left[ F'(u) - F'(u^\varepsilon) \right] u^\varepsilon_i x_i \, dx \\
& = \int_{\partial W_R} \frac{B_\varepsilon'(H(\nabla u^\varepsilon))}{|\nabla H^*|} \langle \nabla H(\nabla u^\varepsilon), \nabla H^* \rangle \langle \nabla u^\varepsilon, x \rangle \, d\H^{n - 1} \\
& \quad - \int_{W_R} B_\varepsilon'(H(\nabla u^\varepsilon)) \langle \nabla H(\nabla u^\varepsilon), \nabla u^\varepsilon \rangle \, dx + \int_{W_R} \left[ F'(u) - F'(u^\varepsilon) \right] \langle \nabla u^\varepsilon, x \rangle \, dx \\
& \ge - \int_{W_R} B_\varepsilon'(H(\nabla u^\varepsilon)) H(\nabla u^\varepsilon) \, dx + \int_{W_R} \left[ F'(u) - F'(u^\varepsilon) \right] \langle \nabla u^\varepsilon, x \rangle \, dx.
\end{aligned}
\end{equation}
Taking the limit as~$\varepsilon \rightarrow 0^+$ in~\eqref{montech34} and~\eqref{montech5}, by~\eqref{BHuconv} we obtain 
\begin{align*}
\int_{\partial W_R} \left[ B(H(\nabla u)) + G(u) \right] H(\nu) \, d\H^{n - 1} & \ge \frac{n}{R} \int_{W_R} B(H(\nabla u)) + G(u) \, dx \\
& \quad - \frac{1}{R} \int_{W_R} B'(H(\nabla u)) H(\nabla u) \, dx.
\end{align*}
By plugging this last identity in~\eqref{montech1} and recalling~\eqref{scaledWen} we finally get
\begin{align*}
\E'(R) & \ge \frac{1}{R^n} \int_{W_R} B(H(\nabla u)) + G(u) - B'(H(\nabla u)) H(\nabla u) \, dx.
\end{align*}
The result now follows since the integral on the right hand side is non-negative by virtue of~\cite[Theorem~1.1]{CFV14}.
\end{proof}

\section{The Liouville-type theorem} \label{liousec}

Here we prove Theorem~\ref{liouthm}. In order to obtain that~$u$ is constant, our first goal is to show that, thanks to the gradient estimate contained in~\cite[Theorem~1.1]{CFV14}, the gradient term in~\eqref{scaledWen} is bounded by the potential. Then, the monotonicity formula of Theorem~\ref{monformthm} and the growth assumption on~$G(u)$ conclude the argument.

The following general result allows us to accomplish the first step.

\begin{lemma} \label{B'tBgeBlem}
Let~$B \in C^2(0, +\infty) \cap C^1([0, +\infty))$ be a function satisfying~\eqref{Bpos} and~$B(0) = B'(0) = 0$. Assume in addition that~$B$ satisfies either~(A)\ensuremath{'} or~(B)\ensuremath{'}. Then, for any~$K > 0$ there exists a constant~$\delta > 0$ such that
\begin{equation} \label{B'tBgeB}
B'(t) t - B(t) \ge \delta B(t),
\end{equation}
for any~$t \in [0, K]$. In particular, under assumption~(A)\ensuremath{'}, inequality~\eqref{B'tBgeB} holds for any~$t \ge 0$.
\end{lemma}
\begin{proof}
We begin by proving~\eqref{B'tBgeB} when~(A)\ensuremath{'} is in force. Since~$B(0) = B'(0) = 0$, we have
$$
B'(t) t - B(t) = \int_0^t B''(s) s \, ds \ge \bar{\gamma} \int_0^t \left( \bar{\kappa} + s \right)^{p - 2} s \, ds.
$$
On the other hand,
$$
B(t) = \int_0^t B'(s) \, ds \le \bar{\Gamma} \int_0^t \left( \bar{\kappa} + s \right)^{p - 2} s \, ds.
$$
By comparing these two expressions, we see that~\eqref{B'tBgeB} holds for any~$t \ge 0$, with~$\delta = \bar{\gamma} / \bar{\Gamma}$.

Then, we deal with case~(B)\ensuremath{'}. Fix~$K > 0$. Being~$B''(0) > 0$ and~$B(0) = B'(0) = 0$, it clearly exist~$\bar{\Gamma} \ge \bar{\gamma} > 0$ such that~$B''(t) \in \left[ \bar{\gamma}, \bar{\Gamma} \right]$,~for any~$t \in [0, K]$. Hence, as before we compute
$$
B'(t) t - B(t) = \int_0^t B''(s) s \, ds \ge \bar{\gamma} \int_0^t s \, ds = \frac{\bar{\gamma}}{2} t^2,
$$
for any~$t \in [0, K]$. Also,
$$
B(t) = \int_0^t \int_0^s B''(\sigma) \, d\sigma ds \le \frac{\bar{\Gamma}}{2} t^2,
$$
for any~$t \in [0, K]$, and again~\eqref{B'tBgeB} is proved.
\end{proof}

\begin{proof}[Proof of Theorem~\ref{liouthm}]
Combining Lemma~\ref{B'tBgeBlem} and~\cite[Theorem~1.1]{CFV14}, we deduce that
\begin{equation} \label{kinpotest}
B(H(\nabla u(x))) \le C G(u(x)) \mbox{ for any } x \in \R^n,
\end{equation}
for some constant~$C > 0$. We stress that, under hypothesis~$(ii)$ of Theorem~\ref{monformthm}, it is crucial that~$\nabla u$ is globally~$L^\infty$ in order to profitably apply Lemma~\ref{B'tBgeBlem}. Recalling the definition~\eqref{scaledWen} of the rescaled energy functional~$\E$, in view of~\eqref{kinpotest} and~\eqref{Gugrowth} we may conclude that
$$
\lim_{R \rightarrow +\infty} \E(R) \le (C + 1) \lim_{R \rightarrow +\infty} \frac{1}{R^{n - 1}} \int_{W_R} G(u(x)) \, dx = 0.
$$
But then, Theorem~\ref{monformthm} tells that~$\E$ is non-decreasing in~$R\in(0, +\infty)$ and, hence, for any~$r > 0$, we have
$$
0 \le \E(r) \le \lim_{R \rightarrow +\infty} \E(R) = 0,
$$
which yields~$\E \equiv 0$. Consequently,~$\nabla u \equiv 0$, i.e.~$u$ is constant.
\end{proof}

\section{On conditions~\eqref{HM} and~\eqref{FK}} \label{char1sec}

In the present section we prove Theorem~\ref{FKcharprop}, thus establishing a characterization of the anisotropies~$H$ which satisfy
\begin{equation} \tag{\ref*{FK}}
\langle H(\xi) \nabla H(\xi), H^*(x) \nabla H^*(x) \rangle = \langle \xi, x \rangle,
\end{equation}
for any~$\xi, x \in \R^n$. Indeed, we show that such requirement is necessary and sufficient for~$H$ to assume the form
\begin{equation} \tag{\ref*{HM}}
H_M(\xi) = \sqrt{\langle M \xi, \xi \rangle},
\end{equation}
for some symmetric and positive definite matrix~$M \in \Mat_n(\R)$.

We begin by showing the necessity of~\eqref{FK}. As a first step towards this aim, we compute the dual function~$H_M^*$.

\begin{lemma} \label{HMlem}
Let~$M \in \Mat_n(\R)$ be symmetric and positive definite. Then,~$H_M^* = H_{M^{-1}}$.
\end{lemma}

\begin{proof}
Being~$M$ positive definite and symmetric, the assignment
$$
\langle \xi, \eta \rangle_M := \langle M \xi, \eta \rangle,
$$
defines an inner product in~$\R^n$. We denote the induced norm by~$\| \cdot \|_M$. Also notice that~$M$ is invertible, so that~$H_{M^{-1}}$ is well defined.

Recalling definition~\eqref{H*def} of dual function and applying the Cauchy-Schwarz inequality to the inner product~$\langle \cdot, \cdot \rangle_M$, we obtain
\begin{align*}
H_M^*(x) & = \sup_{\xi \ne 0} \frac{\langle x, \xi \rangle}{\sqrt{\langle M \xi, \xi \rangle}} = \sup_{\xi \ne 0} \frac{\langle M (M^{-1} x), \xi \rangle}{\sqrt{\langle M \xi, \xi \rangle}} = \sup_{\xi \ne 0} \frac{\langle M^{-1} x, \xi \rangle_M}{\| \xi \|_M} \\
& \le \sup_{\xi \ne 0} \frac{\| M^{-1} x \|_M \| \xi \|_M}{\| \xi \|_M} = \| M^{-1} x \|_M \\
& = \sqrt{\langle M^{-1} x, x \rangle}.
\end{align*}
On the other hand, the choice~$\xi := M^{-1} x$ yields
$$
H_M^*(x) \ge \frac{\langle x, M^{-1} x \rangle}{\sqrt{\langle M M^{-1} x, M^{-1} x \rangle}} = \sqrt{\langle M^{-1} x, x \rangle}.
$$
Hence, recalling definition~\eqref{HM}, the thesis follows.
\end{proof}

With this in hand, we are now able to prove the following

\begin{lemma} \label{FKneclem}
Let~$M \in \Mat_n(\R)$ be a symmetric and positive definite matrix. Then, the norm~$H_M$ satisfies~\eqref{FK}.
\end{lemma}

\begin{proof}
The proof is a simple computation. Notice that for any symmetric~$A \in \Mat_n(\R)$ we have
$$
\partial_i \left( H_A^2(\xi) \right) = \partial_i \left( A_{j k} \xi_j \xi_k \right) = A_{j k} \delta_{j i} \xi_k + A_{j k} \xi_j \delta_{k i} = 2 A_{i j} \xi_j,
$$
for any~$\xi \in \R^n$, $i = 1, \ldots, n$. Thus, we get
$$
H_A(\xi) \partial_i H_A(\xi) = \frac{\partial_i \left( H_A^2(\xi) \right) }{2}= A_{i j} \xi_j.
$$
Applying then Lemma~\ref{HMlem} together with the identity yet obtained with both choices~$A = M$ and~$A = M^{-1}$, we obtain
\begin{align*}
\langle H_M(\xi) \nabla H_M(\xi), H_M^*(\eta) \nabla H_M^*(\eta) \rangle & = \langle H_M(\xi) \nabla H_M(\xi), H_{M^{-1}}(\eta) \nabla H_{M^{-1}}(\eta) \rangle \\
& = M_{i j} \xi_j M_{i k}^{-1} \eta_k \\
& = \delta_{j k} \xi_j \eta_k \\
& = \langle \xi, \eta \rangle,
\end{align*}
which is~\eqref{FK}.
\end{proof}

Now, we prove that the converse implication is also true. Hence,
Theorem~\ref{FKcharprop} will follow. Before addressing the actual proof, we need just another abstract lemma. We believe that the content of the following result will appear somewhat evident to the reader. However, we include both the formal statement and the proof.

\begin{lemma} \label{symlem}
Let~$\T: \R^n \rightarrow \R^n$ be symmetric with respect to the standard inner product in~$\R^n$, that is
\begin{equation} \label{sym}
\langle \T(v), w \rangle = \langle v, \T(w) \rangle,
\end{equation}
for any~$v, w \in \R^n$. Then,~$\T$ is a linear transformation, i.e.
$$
\T(v) = T v \mbox{ for any } v \in \R^n,
$$
for some symmetric~$T \in \Mat_n(\R)$
\end{lemma}

\begin{proof}
The conclusion follows by simply plugging~$w = e_i$ in~$\eqref{sym}$, where~$\{ e_i \}_{i = 1, \ldots, n}$ is the canonical basis in~$\R^n$. Indeed, we have
$$
{[\T(v)]}_i = \langle \T(v), e_i \rangle = \langle v, \T(e_i) \rangle
$$
for any~$v \in \R^n$,~$i = 1, \ldots, n$. Thus we may conclude that~$\T(v) = T v$, where~$T = [T_{i j}]_{i, j = 1, \ldots, n}$ is the matrix with entries
$$
T_{i j} = [\T(e_i)]_j.
$$
The symmetry of~$T$ clearly follows by employing~\eqref{sym} again.
\end{proof}

\begin{proof}[Proof of Theorem~\ref{FKcharprop}]
In view of Lemma~\ref{FKneclem}, it is only left to prove that, under condition~\eqref{FK},~$H$ is forced to be of the form~\eqref{HM}.

By Lemma~\ref{H*reg}, we know that the map~$\Psi_H: \R^n \to \R^n$, defined for~$\xi \in \R^n$ by
$$
\Psi_H(\xi) := H(\xi) \nabla H(\xi),
$$
is invertible with inverse~$\Psi_{H^*}$. Under this notation identity~\eqref{FK} may be read as
\begin{equation} \label{FKcond2}
\langle \Psi_H(\xi), \Psi_{H^*}(\eta) \rangle = \langle \xi, \eta \rangle,
\end{equation}
for any~$\xi, \eta \in \R^n$. Applying~\eqref{FKcond2} with~$\eta = \Psi_H(\zeta)$ we get
$$
\langle  \Psi_H(\xi), \zeta \rangle = \langle  \Psi_H(\xi), \Psi_H^{-1}(\eta) \rangle = \langle  \Psi_H(\xi), \Psi_{H^*}(\eta) \rangle = \langle  \xi, \eta \rangle = \langle \xi, \Psi_H(\zeta) \rangle,
$$
for any~$\xi, \zeta \in \R^n$. That is,~$\Psi_H$ is symmetric with respect to the standard inner product in~$\R^n$ and hence linear, by virtue of Lemma~\ref{symlem}. Therefore, there exists a symmetric~$M \in \Mat_n(\R)$ such that
$$
\nabla {\left( \frac{H^2(\xi)}{2} \right)} = H(\xi) \nabla H(\xi) = M \xi.
$$
This in turn implies that~$H = H_M$ and the proof of the proposition is complete.
\end{proof}

\section{On the weaker assumption~\eqref{FKweak}} \label{char2sec}

In this last section we study the condition
\begin{equation} \tag{\ref{FKweak}}
\sgn \langle H(\xi) \nabla H(\xi), H^*(x) \nabla H^*(x) \rangle = \sgn \langle \xi, x \rangle,
\end{equation}
for any~$\xi, x \in \R^n$, which has been introduced in the statement of Theorem~\ref{monformthm}. First, we have the following general result that provides a simpler equivalent form for assumption~\eqref{FKweak}.

\begin{proposition} \label{propFKweakeq}
Let~$H$ be a~$C^1(\R^n \setminus \{ 0 \})$ be a positive homogeneous function of degree~$1$ satisfying~\eqref{Hpos}. Assume the unit ball~$B_1^H$, as defined by~\eqref{Hball}, to be strictly convex. Then,~\eqref{FKweak} is equivalent to the condition
\begin{equation} \label{FK2zero}
\langle H(\xi) \nabla H(\xi), \eta \rangle = 0 \quad \mbox{ if and only if } \quad \langle \xi, H(\eta) \nabla H(\eta) \rangle = 0,
\end{equation}
for any~$\xi, \eta \in \R^n$.
\end{proposition}
\begin{proof}
First, we remark that, by arguing as in the proof of
Theorem~\ref{FKcharprop}, it is immediate to check that~\eqref{FKweak} can be put in the equivalent form
\begin{equation} \label{wFK2}
\sgn \langle H(\xi) \nabla H(\xi), \eta \rangle = \sgn \langle \xi, H(\eta) \nabla H(\eta) \rangle,
\end{equation}
for any~$\xi, \eta \in \R^n$. Thus, we need to show that~\eqref{FK2zero} is equivalent to~\eqref{wFK2}.

Notice that~\eqref{FK2zero} is trivially implied by~\eqref{wFK2}. Thus, we only need to prove that the converse is also true. To see this, assume~\eqref{FK2zero} to hold and fix~$\xi \in \R^n$. If~$\xi = 0$, then both sides of~\eqref{wFK2} vanish, in view of Lemma~\ref{derBH0}. Suppose therefore~$\xi \ne 0$ and consider the hyperplane
$$
\Pi := \left\{ \eta \in \R^n : \langle H(\xi) \nabla H(\xi), \eta \rangle = 0 \right\},
$$
together with the two half-spaces
$$
\Pi_\pm := \left\{ \eta \in \R^n : \pm \langle H(\xi) \nabla H(\xi), \eta \rangle > 0 \right\}.
$$
By virtue of~\eqref{FK2zero}, the function~$h: \R^n \to \R$, defined by setting
$$
h(\eta) := \langle H(\eta) \nabla H(\eta), \xi \rangle,
$$
vanishes precisely on~$\Pi$. Furthermore, by Lemma~\ref{derBH0} (with~$B(t) = t^2/2$),~$h$ is continuous on the whole of~$\R^n$ and it satisfies
$$
h(\xi) = \langle H(\xi) \nabla H(\xi), \xi \rangle = H^2(\xi) > 0.
$$
But~$\xi \in \Pi_+$, and so~$h$ is positive on~$\Pi_+$, being it connected. Analogously, it holds~$h(-\xi) < 0$ from which we deduce that~$h$ is negative on~$\Pi_-$. Thence,~\eqref{wFK2} follows.
\end{proof}

With the aid of Proposition~\ref{propFKweakeq}, we now restrict to the planar case~$n = 2$ and show that, in this case, all the even anisotropies satisfying~\eqref{FKweak} can be obtained by means of an explicit and operative formula. As a result, it will then become clear that~\eqref{FKweak} is a weaker assumption than~\eqref{FK}.

\begin{proposition} \label{2Dprop}
Let~$r: [0, \pi/2] \to (0, +\infty)$ be a given~$C^2$ function satisfying
\begin{equation} \label{rcond1}
r(\theta) r''(\theta) < 2 r'(\theta)^2 + r(\theta)^2 \mbox{ for a.a. } \theta \in \left[ 0, \frac{\pi}{2} \right],
\end{equation}
and
\begin{equation} \label{rcond2}
r(0) = 1, \qquad r(\pi/2) = r^*, \qquad r'(0) = r'(\pi/2) = 0,
\end{equation}
for some~$r^* \ge 1$. Consider the~$\pi$-periodic function~$\widetilde{r}: \R \to (0, +\infty)$ defined on~$[0, \pi]$ by
\begin{equation} \label{rtildef}
\widetilde{r}(\theta) :=
\begin{dcases}
r(\theta) & \quad {\mbox{if }} 0 \le \theta \le \frac{\pi}{2}, \\
\frac{r^* \sqrt{r(\tau^{-1}(\theta))^2 + r'(\tau^{-1}(\theta))^2}}{
r(\tau^{-1}(\theta))^2} & \quad {\mbox{if }} \frac{\pi}{2} \le 
\theta \le \pi,
\end{dcases}
\end{equation}
where~$\tau: [0, \pi/2] \to [\pi/2, \pi]$ is the bijective map given by
\begin{equation} \label{etaoftheta}
\tau(\eta) = \frac{\pi}{2} + \eta - \arctan \frac{r'(\eta)}{r(\eta)}.
\end{equation}
Then,~$\widetilde{r}$ is of class~$C^1(\R)$, the set
\begin{equation} \label{Cdef}
\left\{ (\rho \cos \theta, \rho \sin \theta) : \rho \in [0, \widetilde{r}(\theta)), \, \theta \in [0, 2 \pi] \right\},
\end{equation}
is strictly convex and its supporting function
$$
\widetilde{H}(\rho \cos \theta, \rho \sin \theta) := \frac{\rho}{\widetilde{r}(\theta)},
$$
defined for~$\rho \ge 0$ and~$\theta \in [0, 2 \pi]$, satisfies~\eqref{FK2zero}.

Furthermore, up to a rotation and a homothety of the plane~$\R^2$, any even positive $1$-homogeneous function~$H \in C^2(\R^2 \setminus \{ 0 \})$ satisfying~\eqref{Hpos}, having strictly convex unit ball~$B_1^H$ and for which~\eqref{FK2zero} holds true is such that~$B_1^H$ is of the form~\eqref{Cdef}, for some positive~$r \in C^2([0, \pi/2])$ satisfying~\eqref{rcond1} and~\eqref{rcond2}. 
\end{proposition}

Before heading to the proof of this proposition, we state the following auxiliary result.

\begin{lemma} \label{qinelem}
Let~$r: [0, \pi/2] \to (0, +\infty)$ be a~$C^2$ function that satisfies condition~\eqref{rcond1} and~$r'(0) = r'(\pi/2) = 0$. Then,
\begin{equation} \label{qine}
- \cot \eta < \frac{r'(\eta)}{r(\eta)} < \tan \eta,
\end{equation}
for any~$\eta \in (0, \pi/2)$.
\end{lemma}

\begin{proof}
For any~$\eta \in (0, \pi / 2)$, we set
$$
q(\eta) := \frac{r'(\eta)}{r(\eta)}
$$
Being the tangent function increasing, we see that the right inequality in~\eqref{qine} is satisfied if and only if
\begin{equation} \label{ftech}
f(\eta) := \arctan q(\eta) < \eta.
\end{equation}
Since
$$
q'(\eta) = \frac{r(\eta) r''(\eta) - r'(\eta)^2}{r(\eta)^2},
$$
we see that, for a.e.~$\eta \in (0, \pi / 2)$,
$$
f'(\eta) = \frac{q'(\eta)}{1 + q(\eta)^2} = \frac{r(\eta) r''(\eta) - r'(\eta)^2}{r(\eta)^2 + r'(\eta)^2} < \frac{r(\eta)^2 + r'(\eta)^2}{r(\eta)^2 + r'(\eta)^2} = 1,
$$
by virtue of~\eqref{rcond1}. Observing that~$f(0) = 0$, we then conclude that
$$
f(\eta) = \int_0^\eta f'(t) \, dt < \eta,
$$
which is~\eqref{ftech}. A similar argument shows that also the left inequality in~\eqref{qine} holds true.
\end{proof}

\begin{proof}[Proof of Proposition~\ref{2Dprop}]
Let~$H \in C^2(\R^2 \setminus \{ 0 \})$ be a given norm. Notice that the boundary of its unit ball~$B_1^H$ may be written in polar coordinates as
$$
\partial B_1^H = \left\{ \gamma(\theta) : \theta \in [0, 2 \pi] \right\},
$$
where
\begin{equation} \label{gammadef}
\gamma(\theta) = (r(\theta) \cos \theta, r(\theta) \sin \theta),
\end{equation}
for some~$\pi$-periodic~$r \in C^2(\R)$. Recall that the curvature of such a curve~$\gamma$ is given by
\begin{equation} \label{kdef}
k(\theta) = \frac{2 r'(\theta)^2 - r(\theta) r''(\theta) + r(\theta)^2}{\left[ r(\theta)^2 + r'(\theta)^2 \right]^{3/2}},
\end{equation}
for any~$\theta \in [0, 2 \pi]$. Hence, hypothesis~\eqref{rcond1} tells us that~$\gamma$ has positive curvature, outside at most a set of zero measure, and, thus, that~$B_1^H$ is strictly convex.

We also remark that condition~\eqref{FK2zero} is equivalent to saying that, for any~$\theta, \eta \in [0, 2 \pi]$,
\begin{equation} \label{FKparall}
\gamma'(\theta) \parallel \gamma(\eta) \quad \mbox{ if and only if } \quad \gamma(\theta) \parallel \gamma'(\eta).
\end{equation}
This can be seen by noticing that~$\nabla H(\gamma(\theta))$ is orthogonal to~$\partial B_1^H$ while~$\gamma'(\theta)$ is tangent.

At a point~$\theta^* \in [0, 2 \pi]$ such that
$$
r(\theta^*) = \max_{\theta \in \R} \, r(\theta) =: r^*,
$$
we clearly have~$r'(\theta^*) = 0$. Assuming, up to a rotation and a homothety of~$\R^2$, that~$\theta^* = \pi / 2$ and~$r(0) = 1$, it is immediate to check, by computing
\begin{equation} \label{gamma'}
\gamma'(\theta) = \left( r'(\theta) \cos \theta - r(\theta) \sin \theta, r'(\theta) \sin \theta + r(\theta) \cos \theta \right),
\end{equation}
that condition~\eqref{FK2zero}, in its form~\eqref{FKparall}, forces~$r$ to satisfy~\eqref{rcond2}.

Now, take~$r \in C^2([0, \pi / 2])$ as in the statement of the proposition. We shall show that the function~$\widetilde{r}$ defined by~\eqref{rtildef} is the only extension of~$r$ which determines a curve~$\gamma$ satisfying condition~\eqref{FKparall}. Notice that, by the periodicity of~$\widetilde{r}$, it is enough to prove it for~$\theta, \eta \in [0, \pi]$. Moreover, if~$\theta, \eta \in \{ 0, \pi/2, \pi \}$, then~\eqref{FKparall} is implied by~\eqref{rcond2}. Consider now~$\eta \in (0, \pi/2)$. We address the problem of finding the unique~$\theta =: \tau(\eta) \in (0, \pi)$ such that~$\gamma(\theta) \parallel \gamma'(\eta)$. First observe that this condition is equivalent to requiring
\begin{equation} \label{cottheta}
\cot \theta = \frac{r'(\eta) \cos \eta - r(\eta) \sin \eta}{r'(\eta) \sin \eta + r(\eta) \cos \eta} = \frac{\frac{r'(\eta)}{r(\eta)} - \tan \eta}{\frac{r'(\eta)}{r(\eta)} \tan \eta + 1} = \tan \left( \arctan \frac{r'(\eta)}{r(\eta)} - \eta \right),
\end{equation}
in view of~\eqref{gammadef} and~\eqref{gamma'}. Then, we see that, by~\eqref{gamma'} and Lemma~\ref{qinelem},~$\gamma'(\eta)$ and, therefore,~$\gamma(\theta)$ lie in the second quadrant. Thus, we conclude that~$\theta \in (\pi / 2, \pi)$. Moreover, with this in hand and using again Lemma~\ref{qinelem}, it is easy to deduce from~\eqref{cottheta} that
\begin{equation} \label{tauofeta}
\theta = \tau(\eta) = \frac{\pi}{2} + \eta - \arctan \frac{r'(\eta)}{r(\eta)},
\end{equation}
for any~$\eta \in \left[ 0, \pi / 2 \right]$. Condition~\eqref{FKparall} then implies that~$\gamma'(\theta) \parallel \gamma(\eta)$, which yields~\eqref{cottheta} with~$\eta$ and~$\theta$ interchanged. Comparing the two formulae, we deduce that~$\widetilde{r}$ should satisfy
\begin{equation} \label{r'overr}
\frac{\widetilde{r}'(\tau(\eta))}{\widetilde{r}(\tau(\eta))} = - \frac{r'(\eta)}{r(\eta)},
\end{equation}
for any~$\eta \in \left[ 0, \pi / 2 \right]$. From this relation it is possible to recover the explicit form of~$\widetilde{r}$. In order to do this, we multiply by~$\tau'(\eta)$ both sides of~\eqref{r'overr} and integrate. The left hand side becomes
\begin{equation} \label{lhsr'overr}
\int_0^\eta \frac{\widetilde{r}'(\tau(t))}{\widetilde{r}(\tau(t))} \tau'(t) \, dt = \log \frac{\widetilde{r}(\tau(\eta))}{\widetilde{r}(\tau(0))} = \log \frac{\widetilde{r}(\tau(\eta))}{r^*}.
\end{equation}
The expansion of the right hand side requires a little bit more care. For simplicity of exposition, we will omit to evaluate~$r$ and its derivatives at~$\eta$. We deduce from~\eqref{tauofeta} that
\begin{equation} \label{tau'}
\tau' = 1 - \frac{r r'' - r'^2}{r^2 + r'^2} = \frac{r^2 + 2 r'^2 - r r''}{r^2 + r'^2}.
\end{equation}
Then, since
$$
\Big[ \log \Big( r \left( r^2 + r'^2 \right) \Big) \Big]' = \frac{3 r^2 r' + r'^3 + 2 r r' r''}{r \left( r^2 + r'^2 \right)},
$$
we compute
\begin{align*}
- \frac{r'}{r} \tau' & = - \frac{r^2 r' + 2 r'^3 - r r' r''}{r \left( r^2 + r'^2 \right)} \\
& = \frac{1}{2} \Big[ \log \Big( r \left( r^2 + r'^2 \right) \Big) \Big]' - \frac{5}{2} \frac{r^2 r' + r'^3}{r \left( r^2 + r'^2 \right)} \\
& = \frac{1}{2} \Big[ \log \Big( r \left( r^2 + r'^2 \right) \Big) - 5 \log r \Big]' \\
& = \frac{1}{2} \left[ \log \frac{r^2 + r'^2}{r^4} \right]'.
\end{align*}
Integrating this last expression we get
\begin{equation} \label{rhsr'overr}
- \int_0^\eta \frac{r'(t)}{r(t)} \tau'(t) \, dt = \frac{1}{2} \log \left( \frac{r(\eta)^2 + r'(\eta)^2}{r(\eta)^4} \frac{r(0)^4}{r(0)^2 + r'(0)^2} \right) = \frac{1}{2} \log \frac{r(\eta)^2 + r'(\eta)^2}{r(\eta)^4}.
\end{equation}
By comparing~\eqref{lhsr'overr} and~\eqref{rhsr'overr}, we immediately obtain that~$\widetilde{r}$ satisfies~\eqref{rtildef}.

Now we show that~$\widetilde{r}$ has the desired regularity properties. From its definition and~\eqref{r'overr} is immediate to see that~$\widetilde{r}$ is continuous on the whole~$[0, \pi]$ and differentiable on~$(0, \pi / 2) \cup (\pi/2, \pi)$. Thus, we only need to check~$\widetilde{r}'$ at~$0$,~$\pi/2$ and~$\pi$. Using~\eqref{r'overr} and~\eqref{rcond2}, we compute
\begin{equation} \label{r'jointN}
\widetilde{r}' \left( {\frac{\pi}{2}}^+ \right) = - \frac{r'(0) \, \widetilde{r} \left( \frac{\pi}{2} \right)}{r(0)} = 0 = \widetilde{r}' \left( {\frac{\pi}{2}}^- \right),
\end{equation}
and
\begin{equation} \label{r'jointEO}
\widetilde{r}'(\pi^-) = - \frac{r' \left( \frac{\pi}{2} \right) \widetilde{r}(\pi)}{r \left( \frac{\pi}{2} \right)} = 0 = \widetilde{r}'(0^+).
\end{equation}
Being it~$\pi$-periodic, it follows that~$\widetilde{r} \in C^1(\R)$.

Finally, we prove that the set~\eqref{Cdef} is strictly convex. To see this, it is enough to show that~$\widetilde{r}$ satisfies~\eqref{rcond1} for almost any~$\theta \in [\pi/2, \pi]$. First, we check that~$\widetilde{r}$ possesses almost everywhere second derivative. Indeed, by differentiating~\eqref{r'overr} we get
\begin{equation} \label{r''impl}
\left( \frac{\widetilde{r}''(\tau(\theta))}{\widetilde{r}(\tau(\theta))} - \frac{\widetilde{r}'(\tau(\theta))^2}{\widetilde{r}(\tau(\theta))^2} \right) \tau'(\theta) = - \frac{r''(\theta)}{r(\theta)} + \frac{r'(\theta)^2}{r(\theta)^2}.
\end{equation}
Thus, if~$\tau'(\theta) \ne 0$, which is true at almost any~$\theta \in [0, \pi / 2]$ in view of~\eqref{tau'} and~\eqref{rcond1}, we may solve~\eqref{r''impl} for~$\widetilde{r}''$ and obtain
\begin{equation} \label{r''expl}
\begin{aligned}
\widetilde{r}''(\tau(\theta)) & = \frac{\widetilde{r}'(\tau(\theta))^2}{\widetilde{r}(\tau(\theta))} - \frac{\widetilde{r}(\tau(\theta))}{\tau'(\theta)} \left( \frac{r''(\theta)}{r(\theta)} - \frac{r'(\theta)^2}{r(\theta)^2} \right) \\
& = \frac{\widetilde{r}'(\tau(\theta))^2}{\widetilde{r}(\tau(\theta))} - \frac{\widetilde{r}(\tau(\theta)) \left( r(\theta)^2 + r'(\theta)^2 \right) \left( r(\theta) r''(\theta) - r'(\theta)^2 \right)}{r(\theta)^2 \left( r(\theta)^2 + 2 r'(\theta)^2 - r(\theta) r''(\theta) \right)},
\end{aligned}
\end{equation}
where in last line we made use of~\eqref{tau'}. With this in hand and recalling~\eqref{r'overr}, we are able to compute that
\begin{align*}
\widetilde{r}(\tau) \widetilde{r}''(\tau) - 2 \widetilde{r}'(\tau)^2 - \widetilde{r}(\tau)^2 & = \widetilde{r}'(\tau)^2 - \frac{\widetilde{r}(\tau)^2 (r^2 + r'^2) (r r'' - r'^2)}{r^2 (r^2 + 2 r'^2 - r r'')} - 2 \widetilde{r}'(\tau)^2 - \widetilde{r}(\tau)^2 \\
& = - \widetilde{r}(\tau)^2 \left( \frac{r'^2}{r^2} + \frac{(r^2 + r'^2) (r r'' - r'^2)}{r^2 (r^2 + 2 r'^2 - r r'')} + 1 \right) \\
& = - \frac{\widetilde{r}(\tau)^2 (r^2 + r'^2)^2}{r^2 (r^2 + 2 r'^2 - r r'')} \\
& < 0,
\end{align*}
almost everywhere in~$[0, \pi / 2]$. Thus, the proof is complete.
\end{proof}

In view of Proposition~\ref{2Dprop}, every even anisotropy~$H$ satisfying~\eqref{FKweak} is uniquely determined by its values on the first quadrant. Conversely, any positive~$r \in C^2([0, \pi / 2])$ for which~\eqref{rcond1} and~\eqref{rcond2} are true can be extended to~$[0, \pi]$ (in a unique way) to obtain a~$C^1$ norm satisfying~\eqref{FKweak}.

An example of such an anisotropy, which is not of the trivial type~\eqref{HM}, is given by
$$
\hat{H}_p(\xi) = \begin{cases}
|\xi|_p & \quad \mbox{if } \xi_1 \xi_2 \ge 0, \\
|\xi|_q & \quad \mbox{if } \xi_1 \xi_2 < 0,
\end{cases}
$$
where~$| \cdot |_p$ is the standard~$p$-norm in~$\R^2$ and~$q = p / (p - 1)$ is the conjugate exponent of~$p$, for~$p \in (2, +\infty)$ (see Figure~\ref{pnormsplot} below). It can be easily checked that~$\hat{H}_p$ satisfies~\eqref{FKweak} from formulation~\eqref{FK2zero}.

\begin{figure}[h]
\centering
\includegraphics[width=\textwidth]{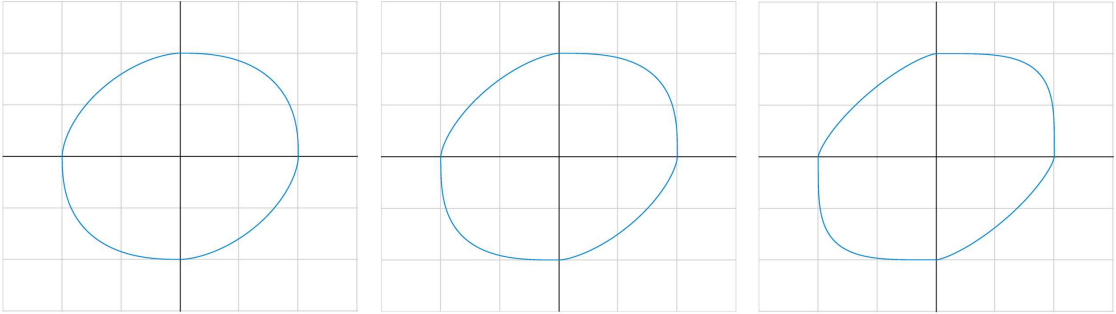}
\vspace{-25pt}
\caption{The unit circles of~$\hat{H}_p$ for the values~$p = 5/2$,~$3$ and~$4$.}
\label{pnormsplot}
\end{figure}

Unfortunately,~$\hat{H}_p$ is no more than~$C^{1, 1 / (p - 1)}_\loc(\R^2 \setminus \{ 0 \})$. If one is interested in norms having higher regularity properties, additional hypotheses on the behaviour of the defining function~$r$ of its unit ball inside the first quadrant need to be imposed. In particular, assumption~\eqref{rcond1} should be strengthened by requiring it to hold at \emph{any}~$\theta \in [0, \pi / 2]$. As a consequence, the class of norms under analysis is restricted to those being uniformly elliptic. 

In order to deal with, say,~$C^{3, \alpha}$ anisotropies, we have the following result.

\begin{proposition} \label{2Dhrprop}
Let~$\alpha \in (0, 1]$ and~$H \in C^{3, \alpha}_\loc(\R^2 \setminus \{ 0 \})$ be an even positive homogeneous function of degree~$1$ for which~\eqref{Hpos} holds true. Then,~$H$ is uniformly elliptic and satisfies~\eqref{FK2zero} if and only if, up to a rotation and a homothety of~$\R^2$, its unit ball is of the form~\eqref{Cdef}, where~$\widetilde{r}$ is given by~\eqref{rtildef} and~$r \in C^{3, \alpha}([0, \pi / 2])$ is a positive function satisfying
\begin{gather} \label{rcond1str}
r(\theta) r''(\theta) < 2 r'(\theta)^2 + r(\theta)^2 \mbox{ for any } \theta \in \left[ 0, \frac{\pi}{2} \right], \\
\label{rcond3}
r'' \left( \frac{\pi}{2} \right) = - \frac{r^* r''(0)}{1 - r''(0)}, \qquad
r''' \left( \frac{\pi}{2} \right) = - \frac{r^* r'''(0)}{(1 - r''(0))^3},
\end{gather}
and~\eqref{rcond2}.
\end{proposition}

Notice that the quantities appearing in both right hand sides of condition~\eqref{rcond3} are finite, as one can see by plugging~$\theta = 0$ in~\eqref{rcond1str} and recalling~\eqref{rcond2}.

\begin{proof}[Proof of Proposition~\ref{2Dhrprop}]
In addition to the regularity properties of the extension~$\widetilde{r}$, by Proposition~\ref{2Dprop} we only need to investigate the relation between~\eqref{rcond1str} and the uniformly convexity of the unit ball of~$H$. Notice that in~$2$ dimensions this last requirement is just asking the curvature~$k(\theta)$, as defined by~\eqref{kdef}, to be positive at any angle~$\theta \in [0, 2 \pi]$. Hence, we see that it implies~\eqref{rcond1str}.

To check that also the converse implication is valid, it is enough to prove that if~\eqref{rcond1str} is in force, then~$\widetilde{r}$ satisfies the same inequality at any~$\theta \in [\pi / 2, \pi]$. A careful inspection of the proof of Proposition~\ref{2Dprop} - see, in particular, the argument starting below formula~\eqref{r''impl} - shows that this is true at any point~$\theta$ for which~$\tau'(\tau^{-1}(\theta)) \ne 0$. But then, comparing formula~\eqref{tau'} with~\eqref{rcond1str} we have that~$\tau' > 0$ on the whole interval~$[0, \pi / 2]$ and so we are done.

The only thing we still have to verify is that, given~$r \in C^{3, \alpha}([0, 2 \pi])$, then its extension~$\widetilde{r}$ belongs to~$C^{3, \alpha}(\R)$. Arguing as in the proof of Proposition~\ref{2Dprop}, by~\eqref{rtildef},~\eqref{r''expl} and~\eqref{rcond1str} we deduce that~$\widetilde{r}$ is of class~$C^1$ on the whole of~$\R$ and $C^{3, \alpha}$ outside of the points~$k \pi / 2$, with~$k \in \Z$. Moreover, by the periodicity properties of~$\widetilde{r}$, we can reduce our analysis to the points~$0$,~$\pi/2$ and~$\pi$.
Using~\eqref{tau'} and~\eqref{rcond2}, we compute
\begin{equation} \label{tau'eval}
\tau'(0) = 1 - r''(0), \qquad \tau' \left( \frac{\pi}{2} \right) = \frac{r^* - r'' \left( \frac{\pi}{2} \right)}{r^*},
\end{equation}
and so, by~\eqref{r''expl},~\eqref{tauofeta},~\eqref{rcond2},~\eqref{r'jointN},~\eqref{r'jointEO} and~\eqref{rcond3}, we have
\begin{align*}
\widetilde{r}'' \left( \frac{\pi}{2}^+ \right) & = \frac{\widetilde{r}' \left( \frac{\pi}{2} \right)^2}{\widetilde{r} \left( \frac{\pi}{2} \right)} - \frac{\widetilde{r} \left( \frac{\pi}{2} \right)}{\tau'(0)} \left( \frac{r''(0)}{r(0)} - \frac{r'(0)^2}{r(0)^2} \right) \\
& = - \frac{r^* r''(0)}{1 - r''(0)} = r'' \left( \frac{\pi}{2} \right) = \widetilde{r}'' \left( \frac{\pi}{2}^- \right),
\end{align*}
and
\begin{align*}
\widetilde{r}''(\pi^-) & = \frac{\widetilde{r}'(\pi)^2}{\widetilde{r}(\pi)} - \frac{\widetilde{r}(\pi)}{\tau' \left( \frac{\pi}{2} \right)} \left( \frac{r'' \left( \frac{\pi}{2} \right)}{r \left( \frac{\pi}{2} \right)} - \frac{r' \left( \frac{\pi}{2} \right)^2}{r \left( \frac{\pi}{2} \right)^2} \right) \\
& = - \frac{r'' \left( \frac{\pi}{2} \right)}{r^* - r'' \left( \frac{\pi}{2} \right)} = r''(0) = \widetilde{r}''(0^+).
\end{align*}
Hence,~$\widetilde{r} \in C^2(\R)$. Now we study the third derivative of~$\widetilde{r}$. By differentiating~\eqref{r''expl} we get
\begin{equation} \label{r'''}
\begin{aligned}
\widetilde{r}'''(\tau) & = \frac{\widetilde{r}'(\tau) \left( 2 \widetilde{r}(\tau) \widetilde{r}''(\tau) - \widetilde{r}'(\tau)^2 \right)}{\widetilde{r}(\tau)^2} \\
& \quad - \frac{\left( \widetilde{r}'(\tau) \tau'^2 - \widetilde{r}(\tau) \tau'' \right) \left( r r'' - r'^2 \right)}{r^2 \tau'^3} - \frac{\widetilde{r}(\tau) \left( r^2 r''' - 3 r r' r'' +2 r'^3 \right)}{r^3 \tau'^2},
\end{aligned}
\end{equation}
where every function is meant to be evaluated at~$\theta$. Moreover, from~\eqref{tau'} we deduce that
\begin{align*}
\tau'' & = - \frac{\left( r' r'' + r r''' - 2 r' r'' \right) \left( r^2 + r'^2 \right) - 2 \left( r r'' - r'^2 \right) \left( r r' + r' r'' \right)}{\left( r^2 + r'^2 \right)^2} \\
& = \frac{3 r^2 r' r'' - r'^3 r'' - r^3 r''' - r r'^2 r''' + 2 r r' r''^2 - 2 r r'^3}{\left( r^2 + r'^2 \right)^2},
\end{align*}
so that, recalling~\eqref{rcond2}, we have
$$
\tau''(0) = - r'''(0), \qquad \tau'' \left( \frac{\pi}{2} \right) = - \frac{r''' \left( \frac{\pi}{2} \right)}{r^*}.
$$
Thus, by plugging these identities into~\eqref{r'''}, using~\eqref{tauofeta},~\eqref{rcond2},~\eqref{tau'eval},~\eqref{r'jointN},~\eqref{r'jointEO} and~\eqref{rcond3} we finally conclude that
\begin{align*}
\widetilde{r}''' \left( \frac{\pi}{2}^+ \right) & = \frac{\widetilde{r} \left( \frac{\pi}{2} \right) \tau''(0) r''(0)}{r(0) \tau'(0)^3} - \frac{\widetilde{r} \left( \frac{\pi}{2} \right) r'''(0)}{r(0) \tau'(0)^2} = - \frac{r^* r''(0) r'''(0)}{\left( 1 - r''(0) \right)^3} - \frac{r^* r'''(0)}{\left( 1 - r''(0) \right)^2} \\
& = - \frac{r^* r'''(0)}{(1 - r''(0))^3} = r''' \left( \frac{\pi}{2} \right) = \widetilde{r}''' \left( \frac{\pi}{2}^- \right),
\end{align*}
and
\begin{align*}
\widetilde{r}'''(\pi^-) & = \frac{\widetilde{r}(\pi) \tau'' \left( \frac{\pi}{2} \right) r'' \left( \frac{\pi}{2} \right)}{r \left( \frac{\pi}{2} \right) \tau' \left( \frac{\pi}{2} \right)^3} - \frac{\widetilde{r} (\pi) r''' \left( \frac{\pi}{2} \right)}{r \left( \frac{\pi}{2} \right) \tau' \left( \frac{\pi}{2} \right)^2} = - \frac{r^* r'' \left( \frac{\pi}{2} \right) r''' \left( \frac{\pi}{2} \right)}{\left( r^* - r'' \left( \frac{\pi}{2} \right) \right)^3} - \frac{r^* r''' \left( \frac{\pi}{2} \right)}{\left( r^* - r'' \left( \frac{\pi}{2} \right) \right)^2} \\
& = - \frac{{r^*}^2 r''' \left( \frac{\pi}{2} \right)}{\left( r^* - r'' \left( \frac{\pi}{2} \right) \right)^3} = r'''(0) = \widetilde{r}'''(0^+).
\end{align*}
As a result,~$\widetilde{r} \in C^{3, \alpha}(\R)$ and the proof of the proposition is complete.
\end{proof}

We observe that Proposition~\ref{2Dprop-i} 
is a consequence of Propositions~\ref{2Dprop}.
and~\ref{2Dhrprop}.

\begin{remark}\label{final}
We point out that it is easy to construct norms which are smooth and satisfy~\eqref{FKweak} as small perturbations of those of the form~\eqref{HM}. For instance, fix any~$\psi \in C^\infty([0, \pi / 2])$ having support compactly contained in~$(0, \pi / 2)$. Then, for~$\varepsilon > 0$ define
$$
r_\psi(\theta) := 1 + \varepsilon \psi(\theta),
$$
for any~$\theta \in \left[ 0, \pi / 2 \right]$. Observe that conditions~\eqref{rcond2} and~\eqref{rcond3} are satisfied with~$r^* = 1$. Moreover, we compute
\begin{align*}
r_\psi r_\psi'' - 2 r_\psi'^2 - r_\psi^2 & = \varepsilon^2 (1 + \varepsilon \psi) \psi'' - 2 \varepsilon^2 \psi'^2 - (1 + \varepsilon \psi)^2 \\
& = - 1 + \varepsilon \left( - 2 \psi + \varepsilon \left( (1 + \varepsilon \psi) \psi'' - 2 \psi'^2 - \psi^2 \right) \right) \\
& \le - 1 + c_\psi \varepsilon,
\end{align*}
with~$c_\psi$ dependent on the~$C^2$ norm of~$\psi$. Therefore, if we take~$\varepsilon$ small enough, then~$r_\psi$ satisfies~\eqref{rcond1str} and, by virtue of Proposition~\ref{2Dhrprop} the associated norm~$H_\psi$ is as desired.
\end{remark}

\end{document}